\documentclass[11pt,reqno,a4paper]{amsart}      
\usepackage{amsmath,amssymb,amsthm,soul,enumerate,enumitem,cancel}
\usepackage{hyperref}
\newtheorem{thm}{Theorem}

\newtheorem{prop}[thm]{Proposition}
\newtheorem{cor}[thm]{Corollary}
\newtheorem{claim}[thm]{Claim}
\newtheorem{lemma}[thm]{Lemma}
\newtheorem{fact}[thm]{Fact}

\theoremstyle{definition}
\newtheorem{defin}[thm]{Definition}
\theoremstyle{remark}
\newtheorem{remark}[thm]{Remark}

\newcommand{\en}{\mathbb N}
\newcommand{\setsep}{:\;}

\newcommand{\Rea}{\mathbb{R}}
\newcommand{\Nat}{\mathbb{N}}
\newcommand{\Rat}{\mathbb{Q}}

\newcommand{\Norm}{\|\cdot \|}

\newcommand{\diam}{\operatorname{diam}}

\newcommand{\Met}{\mathcal{M}}
\newcommand{\Banach}{\mathcal{B}}
\newcommand{\Corr}{\mathcal{R}}

\newcommand{\Span}{\operatorname{span}}
\newcommand{\Net}{\mathcal{N}}
\newcommand{\Lip}{\operatorname{Lip}}
\newcommand{\PP}{\mathcal{P}}

\usepackage{pgf}
\usepackage{tikz}
\usetikzlibrary{arrows,automata}
\usetikzlibrary{positioning}

\begin{document}

\title[Complexity of distances]{Complexity of distances: Reductions of distances between metric and Banach spaces}

\author[M. C\' uth]{Marek C\'uth}
\author[M. Doucha]{Michal Doucha}
\author[O. Kurka]{Ond\v{r}ej Kurka}
\email{cuth@karlin.mff.cuni.cz}
\email{doucha@math.cas.cz}
\email{kurka.ondrej@seznam.cz}

\address[M.~C\' uth, O.~Kurka]{Charles University, Faculty of Mathematics and Physics, Department of Mathematical Analysis, Sokolovsk\'a 83, 186 75 Prague 8, Czech Republic}
\address[M.~Doucha, O.~Kurka]{Institute of Mathematics of the Czech Academy of Sciences, \v{Z}itn\'a 25, 115 67 Prague 1, Czech Republic}



\begin{abstract}
We show that all the standard distances from metric geometry and functional analysis, such as Gromov-Hausdorff distance, Banach-Mazur distance, Kadets distance, Lipschitz distance, Net distance, and Hausdorff-Lipschitz distance have all the same complexity and are reducible to each other in a precisely defined way.

This is done in terms of descriptive set theory and is a part of a larger research program initiated by the authors in \cite{CDKpart1}. The paper is however targeted also to specialists in metric geometry and geometry of Banach spaces.
\end{abstract}
\maketitle

\section*{Introduction}
Metric geometry and nonlinear geometry of Banach spaces are rapidly evolving fields connected to many different areas of mathematics including Riemannian geometry, Banach space theory, graph theory, computer science, etc. One of their feature is that they, as `metric disciplines', quantitatively measure non-equivalence of the objects they work with by distances. Standard examples of such distances are the Gromov-Hausdorff distance between metric spaces and the Banach-Mazur distance between Banach spaces, but the list of useful distances is quite large and the study of those and their mutual relations inspired mathematicians to prove several deep theorems, most notably in the fields of the non-linear geometry of Banach spaces and geometry of Riemannian manifolds, see e.g. \cite{DuKa, G16, GKL00, KalOst, MaPo, O, SoWe}. We have been also inspired by the influential book of Gromov (\cite{Gro}), where many of the distances we work with were defined.

The issue we address in this paper is to compare the complexities, on the small scale, of the distances that have recently received attention in metric geometry and Banach space theory. For instance, we look for a constructive assignment of Banach spaces to metric spaces so that the Banach-Mazur distance of the assigned Banach spaces is small if and only if the Gromov-Hausdorff distance of the original metric spaces is small.

In order to provide such a comparison, given distances $d_1$ and $d_2$ on two classes of metric or Banach spaces $\PP_1$ and $\PP_2$, respectively, we say that $d_1$ is \emph{Borel-uniformly continuous reducible} to $d_2$ if there exists a function $f:\PP_1\to\PP_2$ which is a uniformly continuous embedding with respect to distances $d_1$ and $d_2$, see Definition~\ref{defin:Borel-UnifRedukce} for a more precise treatment.

The following is our main result, which therefore in a (certain precise) sense says that various problems ranging from linear classification of Banach spaces to large scale geometry of metric spaces are comparable in their difficulty.
\begin{thm}\label{thm:intro1}
\begin{enumerate}
\item\label{thm:intro1:(1)} The following pseudometrics are mutually Borel-uniformly continuous bi-reducible: the {\bf Gromov-Hausdorff distance} when restricted to Polish metric spaces, to metric spaces bounded from above, from below, from both above and below, to Banach spaces; the {\bf Banach-Mazur distance} on Banach spaces, the {\bf Lipschitz distance} on Polish metric spaces and Banach spaces; the {\bf Kadets distance} on Banach spaces; the {\bf Hausdorff-Lipschitz distance} on Polish metric spaces; the {\bf net distance}  on Banach spaces.
\item The pseudometrics above are Borel-uniformly continuous reducible to the {\bf uniform} distance on Banach spaces.
\end{enumerate}
\end{thm}
To illustrate the meaning, for example the proof that the Banach-Mazur distance on Banach spaces (denoted by $\rho_{BM}$) is Borel-uniformly continuous reducible to the Lipschitz distance on Polish metric spaces (denoted by $\rho_L$) gives a Borel assignment (that is, a very constructive one, avoiding e.g. the axiom of choice) which assigns to a given Banach space $X$ (separable, infinite-dimensional) a Polish metric space $M(X)$ in such a way that this assignment is a uniformly continuous embedding with respect to pseudometrics $\rho_{BM}$ and $\rho_L$. In particular, $\rho_{BM}(X,Y)=0$ if and only if $\rho_L(M(X),M(Y))=0$. Thus, the problem of whether two Banach spaces are close with respect to the Banach-Mazur distance may be transferred to the problem of whether certain metric spaces are close with respect to the Lipschitz distance.
\bigskip

We recall that there is an active and well established stream within descriptive set theory, called \emph{invariant descriptive set theory} (IDST), whose aim is to provide such reductions for equivalence relations. So for example, it is known that the complexities of the equivalence relations of isomorphism of separable $C^*$-algebras, homeomorphism of metrizable compact spaces, and linear isometry of Banach spaces are the same (see \cite{Sabok},\cite{Zie}, \cite{Melleray}). On the other hand, they are strictly less complex than linear isomorphism of Banach spaces \cite{FLR}, isomorphism of separable operator spaces \cite{ACKKLS}, and strictly more complex than the isomorphism of countable graphs. We refer to \cite{gao} as a general reference. 

We follow this line of thought and our results are also written in the terms of IDST. However, our aim is to make it comprehensible and interesting also for researchers working in metric geometry and geometry of Banach spaces without particular knowledge of descriptive set theory. We must note that this paper naturally complements our paper \cite{CDKpart1}, which is targeted to descriptive set theorists and to which we refer for additional motivation and some general results. However, since the readership of these two articles likely will not be exactly the same, we try to make this paper self-contained.

The paper is organized as follows. In Section~\ref{section:prelim}, we present our examples of distances and prove basic facts about them. We also introduce some basic notions from descriptive set theory and repeat our definitions from \cite{CDKpart1}. In Sections~\ref{sectionReduction1}, \ref{sectionReduction2}, and \ref{sectionReduction3} we concentrate on the constructions of our reductions. The final Section~\ref{section:problems} summarizes our findings and suggests some problems.

\section{Preliminaries and basic results}\label{section:prelim}
The goal of this section is to recall several basic notions from descriptive set theory, such as coding of Polish metric spaces or Banach spaces, and to introduce the distances we work with in this paper. We also prove here several basic results about these distances. The notation and terminology is standard, for the undefined notions see \cite{fhhz} for Banach spaces and \cite{Ke} for descriptive set theory. We only emphasize here that by writting that a mapping is an isometry (or isomorphism) we do not assume it is surjective.

\subsection{Coding of Polish metric spaces and Banach spaces}

We begin with formalizing the class of all infinite Polish metric spaces as a standard Borel space. In most situations it will not be important how we formalize this class, but whenever it does become important we shall use the following definition.

\begin{defin}
By $\Met$ we denote the space of all metrics on $\Nat$. This gives $\Met$ a Polish topology inherited from $\Rea^{\Nat\times\Nat}$.

If $p$ and $q$ are positive real numbers, by $\Met_p$, $\Met^q$ and $\Met_p^q$ respectively, we denote the space of metrics with values in $\{0\}\cup [p,\infty)$, $[0,q]$, and $\{0\}\cup[p,q]$ (assuming that $p<q$), respectively.
\end{defin}
\begin{remark}Every $f\in\Met$ is then a  code for the Polish metric space $M_f$ which is the completion of $(\en,f)$.  Hence, in this sense we may refer to the set $\Met$ as to the standard Borel space of all infinite Polish metric spaces. This approach was used for the first time by Vershik \cite{ver} and further e.g. in \cite{cle}, see also \cite[page 324]{gao}. Another possible approach is to view all Polish metric spaces as the Effros-Borel space of all closed subspaces of a universal space such as the Urysohn space $\mathbb{U}$ (the Effros-Borel space is then denoted by $F(\mathbb{U})$). When one considers the space of all pseudometrics on $\Nat$ then these two approaches are equivalent, see e.g. \cite[Theorem 14.1.3]{gao}. Similarly, one can get a Borel isomorphism $\Theta$ between $\Met$ and $F(\mathbb{U})\setminus F_{fin}(\mathbb{U})$, where $F_{fin}(\mathbb{U})$ denotes the Borel set of finite subsets of $\mathbb{U}$, such that $\Theta(f)$ is isometric to $M_f$ for every $f\in\Met$. Since the Borel set of finite metric spaces is not interesting from our point of view we will ignore it in the sequel.
\end{remark}

\begin{remark}
Let $(M,d)$ be a separable metric space. If there is no danger of confusion, we write $M\in\Met$ by which we mean that the metric $d$ restricted to a countable dense subset of $M$ induces a metric $d'\in\Met$. Analogously, if there is no danger of confusion, we write $M\in \Met_p$, $M\in \Met^q$ or $M\in \Met_p^q$.
\end{remark}

Next, we formalize the class of all infinite-dimensional separable Banach spaces as a standard Borel space. As in the case of infinite Polish metric spaces, the concrete coding of this space is usually not important. However, when we compute that certain
maps from or into this space are Borel we adopt a coding analogous to that one for $\Met$ (and which is more similar to the general coding of metric structures from \cite{BYDNT}).

\begin{defin}\label{defin:spaceOfBanachSpaces}
Let us denote by $V$ the vector space over $\Rat$ of all finitely
supported sequences of rational numbers, that is, the unique infinite-dimensional vector space over $\Rat$ with a countable Hamel basis $(e_n)_{n\in\Nat}$. By $\Banach_0$ we denote the space of all norms on the vector space $V$. This gives $\Banach_0$ a Polish topology inherited from $\Rea^V$. We shall consider only those norms for which its canonical extension to the real vector space $c_{00}$ is still a norm; that is, norms for which the elements $(e_n)_n$ are not only $\Rat$-linearly independent, but also $\Rea$-linearly independent. Let us denote the subset of such norms by $\Banach$. 
\end{defin}
Let us point out that $\Banach$ is a $G_\delta$ subset $\Banach_0$, thus a Polish space of its own. Indeed, it suffices to check that for $\Norm\in \Banach_0$ we have $\Norm\in\Banach$ if and only if for every fixed $n\in\Nat$ the elements $e_1,\ldots,e_n$ are $\Rea$-linearly independent in $(c_{00},\Norm)$, which is an open subset of $\Banach_0$. We show that the complement $C$, the set of those norms in $\Banach_0$ for which the elements $e_1,\ldots,e_n$ are $\Rea$-linearly dependent, is closed in $\Banach_0$. Let $(\Norm_m)_{m\in\Nat}\subseteq C$ converge to $\Norm$. We show that $\Norm\in C$. For each $m\in\Nat$ there are $\alpha^m_1,\ldots,\alpha^m_n\in\Rea$, not all of them zero, such that $\|\sum_{i=1}^n \alpha^m_i e_i\|_m=0$. Without loss of generality, we may assume that $M_m=\max\{|\alpha^m_1|,\ldots,|\alpha^m_n|\}=1$. By passing to a subsequence if necessary, we may therefore assume that each $\alpha^m_i$ converges to some $\alpha_i$, for $i\leq n$, where at least one of the limits is non-zero. It follows that $x=\sum_{i=1}^n \alpha_i e_i\neq 0$ and $\|x\|=0$, showing that $\Norm\in C$.

\begin{remark} Each norm $\nu\in\Banach$ is then a code for an infinite-dimensional Banach space $X_\nu$ which is the completion of $(V,\nu)$. The completion is naturally a complete normed space over $\Rea$. This is the same as taking the canonical extension of $\nu$ to $c_{00}$ and then taking the completion.

We refer the reader to another paper of the authors, \cite{CDDK}, where the space $\Banach$ is thoroughly investigated from the topological point of view, which is however not so relevant for our considerations here, where we are merely satisfied with the fact that $\Banach$ is a standard Borel space.

Hence, we may refer to the set $\Banach$ as to the standard Borel space of all infinite-dimensional separable Banach spaces. Another possible approach, introduced by Bossard \cite{Bos02}, is to view all infinite-dimensional separable Banach spaces as the space $SB(X)$ of all closed linear infinite-dimensional subspaces of a universal separable Banach space $X$; then it is a Borel subset of the Effros-Borel space $F(X)$, the interested reader is referred to the monograph \cite{dodos} for further information. Similarly as in the case of Polish metric spaces, those two approaches are equivalent which is witnessed by Theorem~\ref{theorem:SBiffBanach}. Let us note that even though this is quite a natural approach, the space $\Banach$ and Theorem~\ref{theorem:SBiffBanach} seem to be new.

It would be possible to get a coding of all separable Banach spaces, i.e. even finite-dimensional, if we considered the space of all pseudonorms on $V$. As in the case of Polish metric spaces, the Borel set of all finite-dimensional Banach spaces is not interesting from our point of view, so we will ignore it in the sequel.
\end{remark}

\begin{remark}If there is no danger of confusion, we write $X\in \Banach$ as a shortcut for ``$X$ is an infinite-dimensional separable Banach space''.
\end{remark}

\begin{thm}\label{theorem:SBiffBanach}For every universal separable Banach space $X$, there is a Borel isomorphism $\Theta$ between $\Banach$ and $SB(X)$ such that $\Theta(\nu)$ is isometric to $X_{\nu}$ for every $\nu\in\Banach$.
\end{thm}
\begin{proof}
First, let us observe that whenever $X$ and $Y$ are universal separable Banach spaces, there is a Borel isomorphism $\Phi$ between $SB(X)$ and $SB(Y)$ such that $\Phi(Z)$ is isometric to $Z$ for every $Z\in SB(X)$. Indeed, fix an isometry $i:X\to Y$. Then $SB(X)\ni Z\mapsto i(Z)\in SB(Y)$ defines a Borel injective map, let us call it $\Phi_1$, such that $Z$ is isometric to $\Phi_1(Z)$ for every $Z\in SB(X)$. Next, we find an analogous Borel injective map $\Phi_2:SB(Y)\to SB(X)$. Finally, using the usual proof of the Cantor-Bernstein Theorem (see e.g. \cite[Theorem 15.7]{Ke}), we find a Borel isomorphism $\Phi$ between $SB(X)$ and $SB(Y)$ whose graph lies in the union of the graph of $\Phi_1$ and the inverse of the graph of $\Phi_2$. 

Hence, we may without loss of generality assume that $X = C([0,1])\oplus_2 C([0,1])$. Using the classical Kuratowski-Ryll-Nardzewski principle (see e.g. \cite[Theorem 1.2]{dodos}), we easily get a sequence of Borel maps $d_n:SB(X)\to X$ such that for every $Z\in SB(X)$ the sequence $(d_n(Z))_{n=1}^\infty$ is normalized, linearly independent and linearly dense in $Z$. Indeed, using the Kuratowski-Ryll-Nardzewski principle we find $f_n:SB(X)\to X$ such that $(f_n(Z))_{n=1}^\infty$ is dense in $Z$ for every $Z\in SB(X)$, then we inductively find a sequence $(n_k(Z))_{k\in \Nat}$ satisfying 
\[n_1(Z) = \min\{n\in\Nat\colon f_n(Z)\neq 0\},\] \[n_{k+1}(Z)=\min\{n\in\en\colon f_{n_1(Z)}(Z),\ldots,f_{n_k(Z)}(Z),f_n(Z) \text{ are lin.  independent}\}\]
and finally we put $d_k(Z):=\frac{f_{n_k(Z)}(Z)}{\|f_{n_k(Z)}(Z)\|}$ for $k\in\Nat$, which is easily seen to define a Borel map. Since all uncountable Polish metric spaces are Borel isomorphic, we may pick a Borel isomorphism $j$ between $SB(X)$ and the interval $[1,2]$. Now, we define a Borel injective map $\Theta_1:SB(X)\to \Banach$ by putting for every $Z\in SB(X)$
\[
\Theta_1(Z)(\alpha) = \bigg\|j(Z)\alpha_1 d_{1}(Z) + \sum_{i=2}^\infty \alpha_i d_{i}(Z)\bigg\|,\quad \alpha\in V.
\]
Then $\Theta_1$ is an injective Borel map from $SB(X)$ into $\Banach$ such that $X_{\Theta_1(Z)}$ is isometric to $Z$ for every $Z\in SB(X)$.

Next, by \cite[Lemma 2.4]{ku}, there is a Borel map $\widetilde{\Theta_2}:\Banach\to SB(C([0,1]))$ such that $\widetilde{\Theta_2}(\nu)$ is isometric to $X_\nu$ for every $\nu\in \Banach$. Pick a Borel isomorphism $\gamma$ between $\Banach$ and the interval $[0,1]$ and for every $\nu\in \Banach$ define $\Theta_2(\nu)$ as the Banach space of all $(\gamma(\nu)f,\sqrt{1-\gamma}^2(\nu)f)\in X$ where $f\in \widetilde{\Theta_2}(\nu)$. Then $\Theta_2$ is an injective Borel map from $\Banach$ into $SB(X)$ such that $\Theta_2(\nu)$ is isometric to $X_\nu$ for every $\nu\in\Banach$.

Finally, using the usual proof of the Cantor-Bernstein Theorem (see e.g. \cite[Theorem 15.7]{Ke}), we find a Borel isomorphism $\Theta$ between $\Banach$ and $SB(X)$ whose graph lies in the union of the graph of $\Theta_2$ and the inverse of the graph of $\Theta_1$.
\end{proof}

\subsection{Distances between metric spaces and Banach spaces}\label{section:examplesOfDistances}

\subsubsection{Gromov-Hausdorff distance}
We start with the notion of Gromov-Hausdorff distance which has been investigated already in \cite{CDKpart1}. For the convenience of the reader, we repeat some facts about this distance here.

\begin{defin}[Gromov-Hausdorff distance]\label{defin:GHdistance}
Let $(M,d_M)$ be a metric space and $A,B\subseteq M$ two non-empty subsets. The \emph{Hausdorff distance} between $A$ and $B$ in $M$, $\rho_H^M(A,B)$, is defined as
\[
\max\Big\{\sup_{a\in A} d_M(a,B),\sup_{b\in B} d_M(b,A)\Big\},
\]
where for an element $a\in M$ and a subset $B\subseteq M$, $d_M(a,B)=\inf_{b\in B} d_M(a,b)$.

Suppose now that $M$ and $N$ are two metric spaces. Their \emph{Gromov-Hausdorff distance}, $\rho_{GH}(M,N)$, is defined as the infimum of the Hausdorff distances of their isometric copies contained in a single metric space, that is
\[
\rho_{GH}(M,N)=\inf_{\substack{\iota_M:M\hookrightarrow X\\ \iota_N: N\hookrightarrow X}} \rho_H^X(\iota_M(M),\iota_N(N)),
\]
where $\iota_M$ and $\iota_N$ are isometric embeddings into a metric space $X$.

For two metrics $f, g \in \Met$ we denote by $\rho_{GH}(f,g)$ the Gromov-Hausdorff distance between $(\en,f)$ and $(\en,g)$, which is easily seen to be equal to the Gromov-Hausdorff distance between their completions $M_f$ and $M_g$.
\end{defin}

Let $A$ and $B$ be two sets. A \emph{correspondence} between $A$ and $B$ is a binary relation $\Corr\subseteq A\times B$ such that for every $a\in A$ there is $b\in B$ such that $a\Corr b$, and for every $b\in B$ there is $a\in A$ such that $a\Corr b$.

\begin{fact}[see e.g. Theorem 7.3.25. in \cite{BBI}]\label{fact:GHbyCorrespondences}
Let $M$ and $N$ be two metric spaces. For every $r>0$ we have $\rho_{GH}(M,N)< r$ if and only if there exists a correspondence $\Corr$ between $M$ and $N$ such that $\sup |d_M(m,m')-d_N(n,n')|< 2r$, where the supremum is taken over all $m,m'\in M$ and $n,n'\in N$ with $m\Corr n$ and $m'\Corr n'$.
\end{fact}

It is easier to work with bijections instead of correspondences. One may wonder in which situations we may do so. Let us define the corresponding concept and prove some results in this direction. Those will be used later.

\begin{defin}
By $S_{\infty}$ we denote the set of all bijections from $\Nat$ to $\Nat$. For two metrics on natural numbers $ f, g \in \Met$ and $\varepsilon>0$, we consider the relation
\[
f \simeq_\varepsilon g \quad \Leftrightarrow \quad \exists \pi \in S_{\infty} \, \forall \{ n, m \} \in [\mathbb{N}]^{2} : |f(\pi(n), \pi(m)) - g(n, m)| \leq \varepsilon.
\]
We write $f\simeq g$ if $f\simeq_\varepsilon g$ for every $\varepsilon>0$.
\end{defin}

The following two observations are proved e.g. in \cite[Lemma 15 and Lemma 16]{CDKpart1}.

\begin{lemma}\label{lem:lehciImplikace}
For any two metrics on natural numbers $f, g\in \Met$ and any $\varepsilon>0$ we have $\rho_{GH}(f,g)\leq \varepsilon$ whenever $f\simeq_{2\varepsilon} g$.
\end{lemma}

\begin{lemma}\label{lem:GHequivalenceUniformlyDiscrete}
Let $p>0$ be a real number. For any two metrics on natural numbers $f, g\in \Met_p$ we have $\rho_{GH}(f,g)=\inf\{r\setsep f\simeq_{2r} g\}$ provided that $\rho_{GH}(f,g)< p/2$.
\end{lemma}

Recall that a metric space is perfect if it does not have isolated points.

\begin{lemma}\label{lem:GHPerfectSpaces}
Let $f,g\in\Met$ define two perfect metric spaces.\\ Then $\rho_{GH}(f,g)=\inf\{r\setsep f\simeq_{2r} g\}$.
\end{lemma}
\begin{proof}
By Lemma~\ref{lem:lehciImplikace}, $\rho_{GH}(f,g)\leq r$ whenever $f\simeq_{2r} g$. For the other inequality, suppose $\rho_{GH}(f,g)<r$ and fix $s$ with $\rho_{GH}(f,g) < s < r$. By  Fact \ref{fact:GHbyCorrespondences}, there is a correspondence $\Corr\subseteq \Nat\times \Nat$ witnessing that $\rho_{GH}(f,g)<s$. Now we recursively define a permutation $\pi\in S_\infty$. During the $(2n-1)$-th step of the recursion we ensure that $n$ is in the domain of $\pi$ and during the $2n$-th step we ensure that $n$ is in the range of $\pi$.

Pick an arbitrary $n\in\Nat$ such that $1\Corr n$ and set $\pi(1)=n$. If $n=1$ then we have ensured that $1$ is both in the domain and the range of $\pi$. If $n\neq 1$, then pick some $m\in\Nat$ such that $m\Corr 1$ and set $\pi(m)=1$. If the only integer $m$ with the property that $m\Corr 1$ is equal to $1$, which has been already used, we pick an arbitrary $m'\in\Nat$ that has not been used yet and such that $f(m,m')<r-s$. The existence of such $m'$ follows since $f$ is perfect. We set $\pi(m')=1$. In the general $(2n-1)$-th step we proceed analogously. If $n$ has not been added to the domain of $\pi$ yet we pick some $m$ that has not been added to the range of $\pi$ yet and such that $n\Corr m$. Then we may set $\pi(n)=m$. If there is no such $m$, we pick an arbitrary $m$ such that $n\Corr m$ and take an arbitrary $m'$ with $g(m,m')<r-s$ that has not been added to the range of $\pi$ yet and set $\pi(n)=m'$. The $2n$-th step is done analogously.

When the recursion is finished we claim that for every $n,m$ we have $|f(m,n)-g(\pi(m),\pi(n))|\leq 2r$ which is what we should prove. Suppose e.g. that $\pi(m)$, resp. $\pi(n)$ are such that there are $m'$, resp. $n'$ with $g(m',\pi(m))<r-s$ and $g(n',\pi(n))<r-s$, and $m\Corr m'$ and $n\Corr n'$. The other cases are treated analogously. Then by the choice of $\Corr$ we have
\[\begin{split}
|f(m,n)- g(\pi(m),\pi(n))| & \leq |f(m,n)-g(m',n')| + |g(m',n')-g(\pi(m),n')|\\
&  \quad +  |g(\pi(m),n')-g(\pi(m),\pi(n))| \\
& < 2s + g(m',\pi(m)) + g(n',\pi(n)) < 2r.
\end{split}\]
\end{proof}

\begin{remark}\label{rem:GHPerfectSpaces}Note that although Lemma~\ref{lem:GHPerfectSpaces} is stated only for countable dense subsets of perfect metric spaces, by transfinite recursion it can be proved also for the completions. That is, whenever perfect metric space $M$ and $N$ are such that $\rho_{GH}(M,N)<K$, there exists a bijection $\varphi:M\to N$ such that $|d_M(x,y) - d_N(\varphi(x),\varphi(y))|<2K$ for every $x,y\in M$.
\end{remark}

\begin{remark}
If $f,g\in\Met$ define neither perfect metric spaces, nor do they belong to $\Met_p$, for some $p>0$, then $\simeq_\varepsilon$ does not give good estimates for the Gromov-Hausdorff distance between $f$ and $g$. Consider e.g. $\Nat$ as a metric space with its standard metric and a metric space $C_k=\{m+1/n\setsep m\in\Nat, n\geq k\}\subseteq \Rea$, $k\geq 2$, with a metric inherited from $\Rea$. We have $\rho_{GH}(\Nat,C_k)\to 0$ as $k\to \infty$, but clearly there are no bijections between $\Nat$ and $C_k$ witnessing the convergence.
\end{remark}

\subsubsection{Kadets distance}

\begin{defin}[Kadets distance]\label{defin:Kadetsdistance}
Suppose that $X$ and $Y$ are two Banach spaces. Their \emph{Kadets distance}, $\rho_K(X,Y)$, is defined as the infimum of the Hausdorff distances of their unit balls over all isometric linear embeddings of $X$ and $Y$ into a common Banach space $Z$. That is \[
\rho_K(X,Y)=\inf_{\substack{\iota_X: X\hookrightarrow Z\\ \iota_Y: Y\hookrightarrow Z}} \rho_H^Z(\iota_X(B_X),\iota_Y(B_Y)),
\]
where $\iota_X$ and $\iota_Y$ are linear isometric embeddings into a Banach space $Z$.
\end{defin}

Similarly as the Gromov-Hausdorff distance, the Kadets distance may be expressed in terms of correspondences. First, call a subset $A\subseteq X$ of a real vector space \emph{$\Rat$-homogeneous} if it is closed under scalar multiplication by rationals. The following lemma generalizes \cite[Theorem 2.3]{KalOst}, which uses homogeneous maps. The proof is however very similar.

\begin{lemma}\label{lem:KadetsDescription}
Let $X$ and $Y$ be Banach spaces and $E$ and $F$ be some dense $\Rat$-homogeneous subsets of $X$ and $Y$ respectively. Then we have $\rho_K(X,Y)<\varepsilon$ if and only if there exist $\delta\in(0,\varepsilon)$ and a $\Rat$-homogeneous correspondence $\Corr\subseteq E\times F$ with the property that for every $x\in E$ there is $y\in F$ with $x\Corr y$ and $\|y\|_Y\leq \|x\|_X$, for every $y\in F$ there is $x\in E$ with $x\Corr y$ and $\|x\|_X\leq \|y\|_Y$, and 
\[\begin{split}
\Bigg| \Big\|\sum_{i\leq n} x_i\Big\|_X-\Big\|\sum_{i\leq n} y_i\Big\|_Y\Bigg|\leq (\varepsilon-\delta)\Big(\sum_{i\leq n} \max\{\|x_i\|_X,\|y_i\|_Y\}\Big)
\end{split}\]
for all $(x_i)_i\subseteq E$ and $(y_i)_i\subseteq F$, where for all $i$, $x_i\Corr y_i$.
\end{lemma}
\begin{proof}
If $\rho_K(X,Y)<\varepsilon$, then fix some $\delta\in (0,\varepsilon - \rho_K(X,Y))$ and some isometric embeddings of $X$ and $Y$ into a Banach space $Z$ such that $\rho^Z_H(B_X,B_Y)<\varepsilon-\delta$. Then set $x\Corr y$, for $x\in E$ and $y\in F$, if and only if $\|x-y\|_Z\leq \max\{(\varepsilon-\delta)\|x\|_X,(\varepsilon-\delta)\|y\|_Y\}$.

\medskip
Suppose conversely that we have such $\delta\in (0,\varepsilon)$ and $\Corr\subseteq E\times F$. Set $E'$ to be the linear span of $E$, analogously $F'$ to be the linear span of $F$. Then set $Z=E'\oplus F'$ and define a norm $\Norm_Z$ on $Z$ as follows: for $(x,y)\in Z$ set 
\[\begin{split}
\|(x,y)\|_Z = \inf\bigg\{ & \|x_0\|_X+\|y_0\|_Y+(\varepsilon-\delta)\Big(\sum_{i\leq n} \max\{\|x_i\|_X,\|y_i\|_Y\}\Big)\setsep\\
& x=x_0+\sum_{i\leq n} x_i,\; y=y_0-\sum_{i\leq n} y_i, x_0\in E', y_0\in F',\; x_i\Corr y_i\bigg\}.
\end{split}\]

It is clear that $\Norm_Z$ satisfies the triangle inequality. Moreover, it is easy to check, using the $\Rat$-homogeneity of $\Corr$, that $\Norm_Z$ is $\Rat$-homogeneous, i.e. for every $z\in Z$ and $q\in\Rat$, $\|qz\|_Z=|q|\|z\|_Z$. By continuity, we also get the full homogeneity for all real scalars.

Let us check that for any $x\in E'$ we have $\|(x,0)\|_Z=\|x\|_X$. Clearly, $\|(x,0)\|_Z\leq \|x\|_X$. In order to prove the other inequality, pick $x_0\in E'$, $x_1,\ldots,x_n\in E$, $y_0\in F'$, $y_1,\ldots,y_n\in F$, $x=x_0+\sum_{i\leq n} x_i$, $y=0=y_0-\sum_{i\leq n} y_i$ and $x_i\Corr y_i$. By the assumptions, we have 
\[\begin{split}
\|x_0\|_X & +\|y_0\|_Y+(\varepsilon-\delta)\Big(\sum_{i\leq n} \max\{\|x_i\|_X,\|y_i\|_Y\}\Big)  = \\
& =  \|x_0\|_X + \Big\|\sum_{i\leq n} y_i\Big\|_Y+
(\varepsilon-\delta)\Big(\sum_{i\leq n} \max\{\|x_i\|_X,\|y_i\|_Y\}\Big)\geq\\
& \geq \|x_0\|_X+\Big\|\sum_{i\leq n} x_i\Big\|_X\geq \|x\|_X,
\end{split}\]
as desired. So, the inequality $\|(x,0)\|_Z\geq \|x\|_X$ is proved. Analogously, we show that for every $y\in F'$ we have $\|y\|_Y=\|(0,y)\|_Z$. So $E'$ and $F'$ are isometrically embedded into $Z$. Now for any $x\in B_X\cap E$ by the assumption there is $y\in F$ such that $\|y\|_Y\leq \|x\|_X$ and $x\Corr y$. So $$\|(x,-y)\|_Z\leq (\varepsilon-\delta)\|x\|_X$$ since $x$ can be written as $x_0+x$, where $x_0=0$, and $-y$ as $y_0-y$, where $y_0=0$. Analogously, for every $y\in B_Y\cap F$ there is $x\in E$ such that $\|x\|_X\leq \|y\|_Y$ and $\|(x,-y)\|_Z\leq (\varepsilon-\delta)\|y\|_Y$. Finally we take the completion of $Z$ and get a Banach space $Z'$ to which $X$ and $Y$ linearly isometrically embed so that $\rho_H^{Z'}(B_X,B_Y)<\varepsilon$.
\end{proof}

\subsubsection{Lipschitz distance}

\begin{defin}[Lipschitz distance]\label{defin:LispchitzDistance}
Let $M$ and $N$ be two metric spaces. Their \emph{Lipschitz distance} is defined as
\[
\rho_L(M,N) = \inf\big\{\log\max\{\Lip(T),\Lip(T^{-1})\}\setsep T:M\rightarrow N\text{ is bi-Lipschitz bijection}\big\},
\]
where 
\[
\Lip(T)=\sup_{m\neq n\in M}\frac{d_N(T(m),T(n))}{d_M(m,n)}
\]
is the Lipschitz norm of $T$.
\end{defin}
\begin{remark}\label{rem:lip}
The previous definition of the Lipschitz distance is from \cite[Definition 7.2.1]{BBI}. We note that Gromov in \cite[Definition 3.1]{Gro} defines the Lipschitz distance (between $M$ and $N$) as 
\[
\inf\big\{|\log \Lip(T)|+|\log \Lip(T^{-1})|\setsep T:M\rightarrow N\text{ is bi-Lipschitz}\big\}.
\]
Nevertheless, one can easily check that these two definitions give equivalent distances. Indeed, if we denote by $\rho'_L$ the Lipschitz distance in the sense of Gromov, then we easily see that $$\rho_L\leq \rho'_L\leq 2\rho_L.$$

More differently, Dutrieux and Kalton in \cite{DuKa} define the Lipschitz distance analogously to the definition of the Banach-Mazur distance, which we recall later, as  \footnote{More precisely, they define it without the logarithm which we add in order to satisfy the triangle inequality.}
\[
\inf\big\{\log \Lip(T)\Lip(T^{-1})\setsep T:M\rightarrow N\text{ is bi-Lipschitz}\big\}.
\]
Denote this distance by $\rho''_L$. Clearly, $\rho''_L$ is not equivalent with $\rho_L$ since for example the intervals $[0,1]$ and $[0,2]$ have distance zero only in $\rho''_L$. However, in \cite{DuKa} the authors work mainly with Banach spaces and if $M$ and $N$ are Banach spaces, it is easy to see that we have $\rho''_{L}(M,N)=\rho'_{L}(M,N)$. That follows from the fact that we may consider only those bi-Lipschitz maps such that both $\log \Lip(T)$ and $\log \Lip(T^{-1})$ are non-negative. Indeed, if say $\Lip(T)<1$, then we define $T'=T/\Lip(T)$ and we get $\log \Lip(T)+\log \Lip(T^{-1})=|\log \Lip(T')|+|\log \Lip((T')^{-1})|$.

However, for Banach spaces we have $\rho'_L(M,N)= 2\rho_L(M,N)$. This again follows after the appropriate rescaling of the maps $T:M\rightarrow N$.
\end{remark}
One of the differences between the Gromov-Hausdorff distance and the Lipschitz distance on metric spaces is that for the former if $M$ and $N$ are metric spaces and $M'$, resp. $N'$ their dense subsets, then $\rho_{GH}(M,N)=\rho_{GH}(M',N')$. That an analogous equality does not hold for the Lipschitz distance is witnessed by the following fact. We thank Benjamin Vejnar for providing us an example on which it is based.
\begin{fact}\label{fact:LipchitzNotWitnessedOnDenseSubsets}
There exist metrics $d_M,d_N\in\Met$ on $\Nat$ such that their completions are isometric, however there is no bi-Lipschitz map between $(\Nat,d_M)$ and $(\Nat,d_N)$.
\end{fact}
\begin{proof}
Let $M$ be a Polish metric space, let $G$ be the group of bi-Lipschitz autohomeomorphisms of $M$, and suppose there exists $m\in M$ such that $M\setminus G\cdot m$ is dense in $M$, where $G\cdot m$ is the orbit of $m$ under the action of $G$ on $M$. Let $(x_i)_i$ be some countable dense subset of $M$ such that $\{x_i\setsep i\in\Nat\}\cap G\cdot m=\emptyset$, and let $(y_j)_j$ be another countable dense subset of $M$ such that $y_1=m$. Then there is no bi-Lipschitz map between $(x_i)_i$ and $(y_j)_j$. Indeed, otherwise such a bi-Lipschitz map would extend to some bi-Lipschitz autohomeomorphism $g\in G$ and we would have $g\cdot y_1=g\cdot m=x_k$, for some $k\in\Nat$, which is a contradiction.

To give a simple concrete example, consider $M=[0,1]$ and $m=0$.
\end{proof}

It follows that we cannot in general for $d,p\in\Met$ decide whether $\rho_L(M_d,M_p)<\varepsilon$ just by computing $\rho_L((\Nat,d),(\Nat,p))$. For a correspondence $\Corr\subseteq \Nat^2$ and $n\in\Nat$ we denote by $n\Corr$ the set $\{m\in\Nat\setsep n\Corr m\}$ and by $\Corr n$ the set $\{m\in\Nat\setsep m\Corr n\}$. 
\begin{lemma}\label{lem:LipschitzBySequenceOfCorr}
Let $d,p\in\Met$. Then $\rho_L(M_d,M_p)<r$ if and only if there exists $r' < r$ and  a sequence of correspondences $\Corr_i\subseteq \Nat\times\Nat$ decreasing in inclusion such that
\begin{enumerate}
\item\label{eq:LipschitzBySequenceOfCorr-condition1} for every $\varepsilon > 0$ there exists $i\in\Nat$ such that for every $n\in\Nat$ we have $p\text{-}\diam (n\Corr_i) < \varepsilon$ and $d\text{-}\diam (\Corr_i n) < \varepsilon$;
\item\label{eq:LipschitzBySequenceOfCorr-condition2} for every $i\in\Nat$ and every $n,m,n',m'\in\Nat$ such that $d(n,m)\geq 2^{-i}$ and $n\Corr_i n'$ and $m\Corr_i m'$ we have $p(n',m')\leq \exp(r')d(n,m)$;
\item\label{eq:LipschitzBySequenceOfCorr-condition3} for every $i\in\Nat$ and every $n,m,n',m'\in\Nat$ such that $p(n,m)\geq 2^{-i}$ and $n'\Corr_i n$ and $m'\Corr_i m$ we have $d(n',m')\leq \exp(r')p(n,m)$.
\end{enumerate}
\end{lemma}
\begin{proof}
For the implication from the right to the left, for every $n\in\Nat$ we define $\phi(n)\in M_p$ and $\psi(n)\in M_d$ as the unique element of $\bigcap_i \overline{n\Corr_i}$ and $\bigcap_i \overline{\Corr_i n}$, respectively. We leave to the reader to verify the simple fact that $\phi:\Nat\rightarrow M_p$ is a Lipschitz map with Lipschitz constant less than $\exp(r)$, which therefore extends to a Lipschitz map $\bar \phi: M_d\rightarrow M_p$ with the same Lischitz constant, and if $\bar \psi$ is defined analogously, then $\bar \phi=(\bar \psi)^{-1}$

For the other implication, suppose that we are given a bi-Lipschitz map $\phi: M_d\rightarrow M_p$ such that $L:=\max\{\Lip(\phi),\Lip(\phi^{-1})\} < \exp(r)$ and pick $\varepsilon>0$ with $L+\varepsilon < \exp(r)$. For every $i\in\Nat$, put $\varepsilon_i:=\tfrac{\varepsilon}{i2^{i+1}(1+L)}$ and define a correspondence $\Corr_i$ by
\[
\Corr_i:=\big\{(n,n')\in\Nat\times\Nat\setsep \exists \tilde{n}\in\Nat\quad d(n,\tilde{n}) < \varepsilon_i\;\; \&\;\; p(\phi(\tilde{n}),n') < \varepsilon_i\big\}.
\]
We claim that the correspondences $(\Corr_i)_i$ are as desired. It is easy to see that $\Corr_i\subseteq \Nat\times\Nat$ are correspondences decreasing in inclusion and that \eqref{eq:LipschitzBySequenceOfCorr-condition1} is satisfied. We check condition \eqref{eq:LipschitzBySequenceOfCorr-condition2} and find the number $r'$, the condition \eqref{eq:LipschitzBySequenceOfCorr-condition3} is checked similarly. Fix some $i\in\Nat$ and $n,m,n',m'\in\Nat$ with $n\Corr_i n'$, $m\Corr_i m'$ and $d(n,m)\geq 2^{-i}$. Let $\tilde{n}$ and $\tilde{m}$ be natural numbers witnessing that $n\Corr_i n'$ and $m\Corr_i m'$, respectively. Then we have 
\[\begin{split}
p(n',m') & \leq 2\varepsilon_i + p(\phi(\tilde{n}),\phi(\tilde{m}))\leq 2\varepsilon_i + Ld(\tilde{n},\tilde{m})\\
& \leq 2\varepsilon_i + L(2\varepsilon_i + d(n,m)) = d(n,m)\left(L + \frac{2\varepsilon_i(1+L)}{d(n,m)}\right)\\
& \leq d(n,m)(L + 2^{i+1}\varepsilon_i(1+L)) = d(n,m)(L + \tfrac{\varepsilon}{i}),
\end{split}\]
so if we put $r' = \log(L + \varepsilon)$ we get that \eqref{eq:LipschitzBySequenceOfCorr-condition2} holds and $r' < r$.
\end{proof}

\subsubsection{Banach-Mazur distance}

\begin{defin}[Banach-Mazur distance]\label{defin:BanachMazur}
We recall that if $X$ and $Y$ are Banach spaces, their (logarithmic) \emph{Banach-Mazur distance} is defined as
\[
\rho_{BM}(X,Y) = \inf\big\{\log \|T\|\|T^{-1}\|\setsep T:X\rightarrow Y\text{ is a linear isomorphism}\big\}.
\]
\end{defin}
In contrast to the Lipschitz distance, Banach-Mazur distance can be verified just by looking at isomorphisms that are defined on some fixed countable dense linear subspaces over $\Rat$. That is made precise in the following lemma.\footnote{After proving the lemma, we were told by Gilles Godefroy that a similar statement is already in \cite{Gri03}}

\begin{lemma}\label{lem:BanachMazurVerifiableOnCountableSubspaces}
Let $X$ and $Y$ be separable Banach spaces, let $(e_n)_{n\in\Nat}$ and $(f_n)_{n\in\Nat}$ be linearly independent and linearly dense sequences in $X$ and $Y$, respectively, and put $V=\Rat\Span\{e_n\setsep n\in\Nat\}$, $W=\Rat\Span\{f_n\setsep n\in\Nat\}$.

Then $\rho_{BM}(X,Y)<r$ if and only if there exists a surjective linear isomorphism $T:X\rightarrow Y$ with $\log \|T\|\|T^{-1}\|<r$ and $T(V)=W$.
\end{lemma}

Throughout the proof of the lemma (including the following claim), by an isomorphism we mean a surjective linear isomorphism.

\begin{claim}\label{claim:BMOnCountableSubspaces-firstStep}
Let $ T : X \to Y $ be an isomorphism and $ v_{1}, \dots, v_{n}, v \in V $ be such that $ Tv_{j} \in W $ for $ 1 \leq j \leq n $. Then, given $ \eta > 0 $, there is an isomorphism $ S : X \to Y $ such that
\begin{itemize}
\item $ \Vert S - T \Vert \leq \eta $ and $ \Vert S^{-1} - T^{-1} \Vert \leq \eta $,
\item $ Sv_{j} = Tv_{j} $ for $ 1 \leq j \leq n $,
\item $ Sv \in W $.
\end{itemize}
\end{claim}

\begin{proof}
We consider two cases.

(1) Assume that $ v $ does not belong to the linear span of $ v_{1}, \dots, v_{n} $. In this case, there is $ x^{*} \in X^{*} $ such that $ x^{*}(v) = 1 $ and $ x^{*}(v_{j}) = 0 $ for $ 1 \leq j \leq n $. Let $ \varepsilon > 0 $ be such that $ \varepsilon \leq \eta $, $ \varepsilon < \Vert T^{-1} \Vert^{-1} $, $ (\Vert T^{-1} \Vert^{-1} - \varepsilon)^{-1} \cdot \Vert T^{-1} \Vert \cdot \varepsilon \leq \eta $ and every linear operator $ S : X \to Y $ with $ \Vert S - T \Vert \leq \varepsilon $ is an isomorphism (which is possible, because the set of isomorphisms is open). Let $ w \in W $ be such that $ \Vert w - Tv \Vert \leq \varepsilon/\Vert x^{*} \Vert $, and let
$$ Sx = Tx + x^{*}(x) \cdot (w - Tv), \quad x \in X. $$
Clearly, $ Sv_{j} = Tv_{j} $ for $ j \leq n $ and $ Sv = w \in W $. At the same time, $ \Vert S - T \Vert \leq \Vert x^{*} \Vert \Vert w - Tv \Vert \leq \varepsilon \leq \eta $. Note that $ \Vert Sx \Vert \geq \Vert Tx \Vert - \varepsilon \Vert x \Vert \geq (\Vert T^{-1} \Vert^{-1} - \varepsilon) \Vert x \Vert $ for $ x \in X $, and that $ S $ is an isomorphism with $ \Vert S^{-1} \Vert \leq (\Vert T^{-1} \Vert^{-1} - \varepsilon)^{-1} $ in particular. Finally, we obtain $ \Vert S^{-1} - T^{-1} \Vert = \Vert S^{-1} (T - S) T^{-1} \Vert \leq \Vert S^{-1} \Vert \Vert T - S \Vert \Vert T^{-1} \Vert \leq (\Vert T^{-1} \Vert^{-1} - \varepsilon)^{-1} \cdot \varepsilon \cdot \Vert T^{-1} \Vert \leq \eta $.

(2) Assume that, on the other hand, $ v $ belongs to the linear span of $ v_{1}, \dots, v_{n} $. We just need to check that $ v $ belongs to the $ \mathbb{Q} $-linear span of $ v_{1}, \dots, v_{n} $ as well, since then clearly $ Tv \in W $ and the choice $ S = T $ works. There are a large enough $ m \in \mathbb{N} $ and rational numbers $ q^{i}, q_{j}^{i} $ such that
$$ v = \sum_{i=1}^{m} q^{i} e_{i}, \quad v_{j} = \sum_{i=1}^{m} q_{j}^{i} e_{i}. $$
For some real numbers $ \alpha_{1}, \dots, \alpha_{n} $, we have $ v = \sum_{j=1}^{n} \alpha_{j} v_{j} $. That is,
$$ \sum_{i=1}^{m} q^{i} e_{i} = \sum_{j=1}^{n} \alpha_{j} \sum_{i=1}^{m} q_{j}^{i} e_{i} = \sum_{i=1}^{m} \Big( \sum_{j=1}^{n} q_{j}^{i} \alpha_{j} \Big) e_{i}. $$
As $ e_{1}, \dots, e_{n} $ are assumed to be linearly independent, we obtain
$$ \sum_{j=1}^{n} q_{j}^{i} \alpha_{j} = q^{i}, \quad i = 1, \dots, m. $$
Hence, the system of linear equations $ \sum_{j=1}^{n} q_{j}^{i} x_{j} = q^{i}, i = 1, \dots, m, $ has a solution. It follows from the methods of solving systems of linear equations that it has a solution $ \beta_{1}, \dots, \beta_{n} $ consisting of rational numbers. By a similar computation as above, we can obtain $ v = \sum_{j=1}^{n} \beta_{j} v_{j} $.
\end{proof}

\begin{proof}[Proof of Lemma~\ref{lem:BanachMazurVerifiableOnCountableSubspaces}]
Let $ T_{0} : X \to Y $ be an isomorphism with $ \Vert T_{0} \Vert \Vert T_{0}^{-1} \Vert < e^{r} $. Let us pick a small enough $ \varepsilon > 0 $ such that $ (\Vert T_{0} \Vert +\varepsilon) (\Vert T_{0}^{-1} \Vert + \varepsilon) < e^{r} $. We are going to find sequences $ T_{1}, T_{2}, \dots $ of isomorphisms, $ x_{1}, x_{2}, \dots $ of points in $ V $ and $ y_{1}, y_{2}, \dots $ of points in $ W $ such that
\begin{itemize}
\item $ \Vert T_{k} - T_{k-1} \Vert \leq 2^{-k}\varepsilon $ and $ \Vert T_{k}^{-1} - T_{k-1}^{-1} \Vert \leq 2^{-k}\varepsilon $,
\item $ T_{k}e_{j} = y_{j} $ and $ T_{k}^{-1}f_{j} = x_{j} $ for $ j \leq k $.
\end{itemize}

Let us assume that $ k \in \mathbb{N} $ and that we have already found $ T_{j} $, $ x_{j} $ and $ y_{j} $ for $ j < k $. Applying Claim~\ref{claim:BMOnCountableSubspaces-firstStep}, we obtain an isomorphism $ \tilde{T}_{k-1} : X \to Y $ such that
\begin{itemize}
\item $ \Vert \tilde{T}_{k-1} - T_{k-1} \Vert \leq 2^{-k-1}\varepsilon $ and $ \Vert \tilde{T}_{k-1}^{-1} - T_{k-1}^{-1} \Vert \leq 2^{-k-1}\varepsilon $,
\item $ \tilde{T}_{k-1}e_{j} = T_{k-1}e_{j} $ for $ j < k $ and $ \tilde{T}_{k-1}x_{j} = T_{k-1}x_{j} $ for $ j < k $,
\item $ \tilde{T}_{k-1}e_{k} \in W $.
\end{itemize}
Let us put $ y_{k} = \tilde{T}_{k-1}e_{k} $. Applying Claim~\ref{claim:BMOnCountableSubspaces-firstStep} once more, we obtain an isomorphism $ S_{k} : Y \to X $ such that
\begin{itemize}
\item $ \Vert S_{k} - \tilde{T}_{k-1}^{-1} \Vert \leq 2^{-k-1}\varepsilon $ and $ \Vert S_{k}^{-1} - \tilde{T}_{k-1} \Vert \leq 2^{-k-1}\varepsilon $,
\item $ S_{k}f_{j} = \tilde{T}_{k-1}^{-1}f_{j} $ for $ j < k $ and $ S_{k}y_{j} = \tilde{T}_{k-1}^{-1}y_{j} $ for $ j \leq k $,
\item $ S_{k}f_{k} \in V $.
\end{itemize}
Let us put $ x_{k} = S_{k}f_{k} $ and $ T_{k} = S_{k}^{-1} $. Let us check that the choice works. We have $ \Vert T_{k} - T_{k-1} \Vert = \Vert S_{k}^{-1} - T_{k-1} \Vert \leq \Vert S_{k}^{-1} - \tilde{T}_{k-1} \Vert + \Vert \tilde{T}_{k-1} - T_{k-1} \Vert \leq 2^{-k-1}\varepsilon + 2^{-k-1}\varepsilon = 2^{-k}\varepsilon $ and $ \Vert T_{k}^{-1} - T_{k-1}^{-1} \Vert = \Vert S_{k} - T_{k-1}^{-1} \Vert \leq \Vert S_{k} - \tilde{T}_{k-1}^{-1} \Vert + \Vert \tilde{T}_{k-1}^{-1} - T_{k-1}^{-1} \Vert \leq 2^{-k-1}\varepsilon + 2^{-k-1}\varepsilon = 2^{-k}\varepsilon $. For $ j < k $, we have $ T_{k}e_{j} = S_{k}^{-1} \tilde{T}_{k-1}^{-1} \tilde{T}_{k-1} e_{j} = S_{k}^{-1} \tilde{T}_{k-1}^{-1} T_{k-1} e_{j} = S_{k}^{-1} \tilde{T}_{k-1}^{-1} y_{j} = S_{k}^{-1} S_{k} y_{j} = y_{j} $ and $ T_{k}^{-1}f_{j} = S_{k}f_{j} = \tilde{T}_{k-1}^{-1}f_{j} = \tilde{T}_{k-1}^{-1} T_{k-1} T_{k-1}^{-1} f_{j} = \tilde{T}_{k-1}^{-1} T_{k-1} x_{j} = \tilde{T}_{k-1}^{-1} \tilde{T}_{k-1} x_{j} = x_{j} $. Finally, $ T_{k}e_{k} = S_{k}^{-1} \tilde{T}_{k-1}^{-1} \tilde{T}_{k-1} e_{k} = S_{k}^{-1} \tilde{T}_{k-1}^{-1} y_{k} = S_{k}^{-1} S_{k} y_{k} = y_{k} $ and $ T_{k}^{-1}f_{k} = S_{k}f_{k} = x_{k} $.

So, the sequences $ T_{k} $, $ x_{k} $ and $ y_{k} $ are found. Clearly, the sequence $ T_{0}, T_{1}, \dots $ is Cauchy and has a limit $ T $ with $ \Vert T - T_{0} \Vert \leq \sum_{k=1}^{\infty} 2^{-k}\varepsilon = \varepsilon $. Similarly, the sequence $ T_{0}^{-1}, T_{1}^{-1}, \dots $ has a limit $ S $ with $ \Vert S - T_{0}^{-1} \Vert \leq \varepsilon $. Moreover, $ TS = \lim_{k \to \infty} T_{k}T_{k}^{-1} = \lim_{k \to \infty} I = I $, and so $ T $ is an isomorphism with $ T^{-1} = S $. It follows that
$$ \Vert T \Vert \Vert T^{-1} \Vert \leq (\Vert T_{0} \Vert +\varepsilon) (\Vert T_{0}^{-1} \Vert + \varepsilon) < e^{r}. $$
At the same time, $ Te_{j} = y_{j} \in W $ and $ T^{-1}f_{j} = x_{j} \in V $ for every $ j $. Hence, we arrive at $ T(V) = W $.
\end{proof}

The last three distances we shall present are all related to the coarse (or large scale) geometry of metric (and Banach) spaces. We refer the reader to \cite[Chapter 8]{BBI} or the monograph \cite{NoYu} for an introduction into this subject.
\subsubsection{Hausdorff-Lipschitz and net distances}

Gromov defines in \cite[Definition 3.19]{Gro} a distance defined as some variation of both the Gromov-Hausdorff and Lipschitz distances.

\begin{defin}[Hausdorff-Lipschitz distance]\label{defin:HausdorffLipschitz}
For metric spaces $M$ and $N$, their \emph{Hausdorff-Lipschitz distance} is defined as 
\[
\rho_{HL}(M,N)=\inf \big\{\rho_{GH}(M,M')+\rho_L(M',N')+\rho_{GH}(N',N)\setsep M',N'\text{ metric spaces}\big\}.
\]
\end{defin}
The Hausdorff-Lipschitz distance corresponds to the notion of \emph{quasi-isometry} or \emph{coarse Lipschitz equivalence}, because for metric spaces $M$ and $N$ we have $\rho_{HL}(M,N)<\infty$ if and only if the spaces $M$ and $N$ are quasi-isometric, or coarse Lipschitz equivalent (see e.g. \cite[Section 8.3]{BBI} for further information). For information about coarse geometry of Banach spaces we refer to the survey \cite{Lan} or the monograph \cite{Ostbook}.

\medskip
Following \cite[Definition 10.18]{BenLin}, by an \emph{$(a,b)$-net} in a metric space $M$, where $a,b$ are positive reals, we mean a subset $\Net\subseteq M$ such that for every $m\neq n\in\Net$ we have $d(m,n)\geq a$, and for every $x\in M$ there exists $n\in\Net$ with $d(x,n)<b$. If the constants $a$ and $b$ are not important, we just call the subset $\Net$ \emph{a net}. Observe that a maximal $\varepsilon$-separated subset $\Net\subseteq M$ (which exists by Zorn's lemma) is an $(\varepsilon,\varepsilon)$-net. Dutrieux and Kalton \cite{DuKa} consider the net distance which we define as follows (let us note that a slightly different definition of $\rho_L$ is used in \cite{DuKa}).

\begin{defin}[Net distance]\label{defin:netDistance}
The \emph{net distance} between two Banach spaces $X$ and $Y$ is defined as \[
\rho_N(X,Y) = \inf \big\{\rho_L(\Net_X,\Net_Y)\setsep \Net_X,\Net_Y\text{ are nets in }X,Y\text{ respectively}\big\}.
\]
\end{defin}

The next observation is in a sense a quantitative version, for Banach spaces, of \cite[Proposition 8.3.4]{BBI}, where it is proved that two metric spaces are quasi-isometric if and only if they have Lipschitz equivalent nets.

\begin{prop}\label{prop:HLandNetDistanceAreEqual}
For Banach spaces $X$ and $Y$ we have $\rho_N(X,Y)=\rho_{HL}(X,Y)$.
\end{prop}
\begin{proof}
Fix Banach spaces $X$ and $Y$ and a positive real $K$. Suppose that $\rho_N(X,Y)<K$. So there exist an $(a,b)$-net $\Net_X\subseteq X$ and an $(a',b')$-net $\Net_Y\subseteq Y$ and a bi-Lipschitz map $T:\Net_X\rightarrow \Net_Y$ with $\log\max\{ \Lip(T), \Lip(T^{-1})\}<K$. Take any $\varepsilon>0$. By rescaling the nets $\Net_X$ and $\Net_Y$ by a sufficiently large constant $C$ if necessary, that is, taking $\Net_X/C=\{x/C\setsep x\in\Net_X\}$ and $\Net_Y/C$, we may suppose that the nets $\Net_X$ and $\Net_Y$ are an $(a,\varepsilon)$-net, resp. an $(a',\varepsilon)$-net. Then we clearly have $\rho_{GH}(X,\Net_X)\leq \varepsilon$ and $\rho_{GH}(Y,\Net_Y)\leq\varepsilon$, so $$\rho_{HL}(X,Y)\leq \rho_{GH}(X,\Net_X)+\rho_L(\Net_X,\Net_Y)+\rho_{GH}(\Net_Y,Y)< K+2\varepsilon.$$ Since $\varepsilon>0$ was arbitrary, it shows that $\rho_{HL}(X,Y)\leq K$.

Conversely, suppose that $\rho_{HL}(X,Y)<K$. So there exist metric spaces $X'$ and $Y'$ such that $\rho_{GH}(X,X')+\rho_L(X',Y')+\rho_{GH}(Y',Y)<K$. By Fact \ref{fact:GHbyCorrespondences} there are correspondences $\Corr_X\subseteq X\times X'$ and $\Corr_Y\subseteq Y'\times Y$ witnessing that $\rho_{GH}(X,X') < K$ and $\rho_{GH}(Y',Y) < K$. Let $C>0$ be a sufficiently large constant, more precisely specified later, and find some $C$-maximal separated set $\Net_X$ in $X$, which is therefore a $(C,C)$-net. Since $C$ is large, for every $n\neq m\in\Net_X$ we have that $\{x\in X'\setsep n\Corr_X x\}\cap \{x\in X'\setsep m\Corr_X x\}=\emptyset$, so we pick some injective map $f_1:\Net_X\rightarrow X'$ such that for every $n\in\Net_X$ we have $n\Corr_X f_1(n)$. Since $\rho_L(X',Y')<K$ there exists a bi-Lipschitz map $T:X'\rightarrow Y'$ with $\max\{\Lip(T),\Lip(T^{-1})\}< \exp(K)$. Again since $C$ is large enough it follows that for every $n\neq m\in\Net_X$ we have that $\{y\in Y\setsep (T\circ f_1)(n)\Corr_Y y\}\cap \{y\in Y\setsep (T\circ f_1)(m)\Corr_Y y\}=\emptyset$, so we pick some injective map $f_2: (T\circ f_1)[\Net_X]\rightarrow Y$ such that for every $z\in (T\circ f_1)[\Net_X]$ we have $z\Corr_Y f_2(z)$. Set $\phi=f_2\circ T\circ f_1: \Net_X\rightarrow Y$. It follows that the range of $\phi$ is a net $\Net_Y$ in $Y$. Let us compute the Lipschitz constant of $\phi$ and $\phi^{-1}$. For any $n\neq m\in\Net_X$ we have 
\[\begin{split}
\|\phi(n)-\phi(m)\|_Y & \leq d_{Y'}((T\circ f_1)(n),(T\circ f_1)(m))+2K \\
& < \exp(K)d_{X'}(f_1(n),f_1(m)) + 2K \\
& \leq \exp(K)(\|n-m\|_X + 2K) + 2K\\
& \leq \Big(\exp(K) + \tfrac{2K(\exp(K) + 1)}{C}\Big)\|n-m\|_X.
\end{split}\]
However, $\frac{2K(\exp(K) + 1)}{C}\to 0$ as $C\to \infty$. The computation of $\Lip(\phi^{-1})$ is analogous, so we get that $\rho_N(X,Y)\leq K$, and we are done.
\end{proof}
\begin{remark}\label{remark:ConePropertyOfBanach}
Note that in Proposition \ref{prop:HLandNetDistanceAreEqual} the only geometric property of Banach spaces that we used in the proof is that any rescaling of a Banach space $X$ is isometric to $X$. Spaces with this property are called \emph{cones} \cite[Definition 8.2.1]{BBI}. So we have proved that if $\rho_N$ was defined in an obvious way on metric spaces, it would coincide with $\rho_{HL}$ on cones.
\end{remark}

Our next result shows that it is possible to express the Hausdorff-Lipschitz distance, up to uniform equivalence, in terms of correspondences. This observation will be used further.

\begin{defin}Let $d,e\in\Met$ and $\varepsilon > 0$. We say that $d$ and $e$ are \emph{$HL(\varepsilon)$-close} if there exists a correspondence $\Corr\subseteq \Nat\times\Nat$ such that for every $i,i',j,j'\in\Nat$ with $i\Corr j$ and $i'\Corr j'$ we have
\begin{align}
\label{eq:HL1} e(j,j')\leq d(i,i') + \varepsilon \cdot \max \{ 1, d(i,i') \},\\
\label{eq:HL2} d(i,i')\leq e(j,j') + \varepsilon \cdot \max \{ 1, e(j,j') \}.
\end{align}
\end{defin}

\begin{lemma}\label{lem:HLByCorrespondences}There are continuous functions $\varphi_i:(0,\infty)\to (0,\infty)$, $i\in\{1,2\}$, such that $\lim_{\varepsilon\to 0} \varphi_i(\varepsilon) = 0$ and, whenever $d,e\in\Met$ and $\varepsilon > 0$ are given, we have
\[\begin{split}
\rho_{HL}(d,e) < \varepsilon &  \Rightarrow \text{$d$ and $e$ are $HL(\varphi_1(\varepsilon))$-close};\\
\text{$d$ and $e$ are $HL(\varepsilon)$-close} & \Rightarrow \rho_{HL}(d,e) < \varphi_2(\varepsilon).
\end{split}\]
\end{lemma}
\begin{proof}First, let us assume that $\rho_{HL}(d,e) < \varepsilon$, that is, there are $d',e'\in\Met$ with $\rho_{GH}(d,d') + \rho_{L}(d',e') + \rho_{GH}(e',e) < \varepsilon$. By Fact~\ref{fact:GHbyCorrespondences}, there are correspondences $\Corr_1\subseteq \Nat\times\Nat$ and $\Corr_3\subseteq\Nat\times\Nat$ witnessing that $\rho_{GH}(d,d')< \varepsilon$ and $\rho_{GH}(e',e) < \varepsilon$. Further, let $f:M_{d'}\to M_{e'}$ be a bi-Lipschitz bijection witnessing that $\rho_{L}(d',e') < \varepsilon$. Consider now the correspondence
\[\begin{split}
\Corr := \big\{(i,j)\in\Nat\times\Nat \setsep & \text{there are } k,l\in\Nat\text{ such that }(i,l)\in\Corr_1, (k,j)\in\Corr_3, \\
& d'(l,f^{-1}(k)) < \varepsilon \text{ and } e'(f(l),k) < \varepsilon\big\}.
\end{split}\]
This is indeed a correspondence since given $i\in\Nat$ we find $l$ with $(i,l)\in\Corr_1$, pick $k\in\Nat$ with $e'(f(l),k) < \min\{\varepsilon,\tfrac{\varepsilon}{\Lip(f^{-1})}\}$ and find $j$ with $(k,j)\in\Corr_3$; thus, we have $(i,j)\in\Corr$ and similarly for every $j\in\Nat$ there is $i$ with $(i,j)\in\Corr$.

Fix $i,i',j,j'\in\Nat$ with $i\Corr j$ and $i'\Corr j'$. Then there are $l,l'\in \Nat$ and $k,k'\in \Nat$ with $i\Corr_1 l$, $i'\Corr_1 l'$, $k\Corr_3 j$, $k'\Corr_3 j'$, $d'(l,f^{-1}(k)) < \varepsilon$, $e'(f(l),k) < \varepsilon$, $d'(l',f^{-1}(k')) < \varepsilon$ and $e'(f(l'),k') < \varepsilon$. We have
\[\begin{split}
d(i,i') & \leq d'(l,l') + 2\varepsilon\leq d'(f^{-1}(k),f^{-1}(k')) + 4\varepsilon\leq \Lip(f^{-1})e'(k,k') + 4\varepsilon\\
& \leq \exp(\varepsilon)(e(j,j') + 2\varepsilon) + 4\varepsilon\\
& = e(j,j') + \big(\exp(\varepsilon) - 1\big)e(j,j') + 2\varepsilon\big(\exp(\varepsilon) + 2\big).
\end{split}\]
By symmetry, similar inequality holds when the roles of $d$ and $e$ are changed. Hence, if $\varphi_1(\varepsilon) = \exp(\varepsilon) - 1 + 2\varepsilon\exp(\varepsilon) + 4\varepsilon$, then $d$ and $e$ are $HL(\varphi_1(\varepsilon))$-close.

Conversely, let $\Corr\subseteq \Nat\times\Nat$ be a correspondence witnessing that $d$ and $e$ are $HL(\varepsilon)$-close. Put $\delta = \varepsilon + \sqrt{\varepsilon}$. Let $\mathcal{N}_{d}$ be a maximal $\delta$-separated set in $(\Nat,d)$. For every $i\in\mathcal{N}_d$, we pick some $r(i)\in\Nat$ such that $i\Corr r(i)$. Then we put $\mathcal{N}_e:=\{r(i)\setsep i\in\mathcal{N}_d\}$. Clearly,
\[
\rho_{GH}((\Nat,d),\mathcal{N}_d)\leq\delta.
\]
We claim that for every $j\in\Nat$ there is $j'\in\mathcal{N}_e$ with $e(j,j') < \delta + \varepsilon\cdot\max\{1,\delta\}$, which gives
\[
\rho_{GH}((\Nat,e),\mathcal{N}_e)\leq \delta + \varepsilon\cdot\max\{1,\delta\}.
\]
Indeed, if $j\in\Nat$ is given, there is $i$ with $i\Corr j$. Pick $i'\in\mathcal{N}_d$ with $d(i,i') < \delta$. Using \eqref{eq:HL1}, we obtain $e(j,r(i')) < \delta + \varepsilon\cdot\max\{1,\delta\}$. 

Now, let us compute the Lipschitz constant for $r$ and $r^{-1}$. Consider $i,i'\in\mathcal{N}_{d}$, $i\neq i'$. If $d(i,i')\geq 1$, by \eqref{eq:HL1}, we get $e(r(i),r(i'))\leq (1+\varepsilon)d(i,i')$. If $d(i,i')\leq 1$, by \eqref{eq:HL1} and using that $\mathcal{N}_d$ is $\delta$-separated, we get $e(r(i),r(i'))\leq d(i,i') + \varepsilon \leq (1+\tfrac{\varepsilon}{\delta})d(i,i')$. Hence, $\Lip(r)\leq \max\{1+\varepsilon,1+\tfrac{\varepsilon}{\delta}\}$. Note that for every $k,k'\in\mathcal{N}_d$, $k\neq k'$, with $e(r(k),r(k'))\leq 1$, by \eqref{eq:HL2}, we have $e(r(k),r(k'))\geq d(k,k') - \varepsilon\geq \delta - \varepsilon$; hence, similar computation gives $\Lip(r^{-1})\leq \max\{1+\varepsilon,1 + \tfrac{\varepsilon}{\delta - \varepsilon}\} = 1 + \max\{\varepsilon,\sqrt{\varepsilon}\}$. Thus, we have $\rho_{L}(\mathcal{N}_d,\mathcal{N}_e)\leq \log(1 + \max\{\varepsilon,\sqrt{\varepsilon}\})$. Finally, if 
\[
\varphi_2(\varepsilon) = 2\varepsilon + 2\sqrt{\varepsilon} + \log(1 + \max\{\varepsilon,\sqrt{\varepsilon}\}) + \varepsilon\cdot\max\{1,\varepsilon + \sqrt{\varepsilon}\},
\]
we get
\[
\rho_{HL}(d,e)\leq \delta + \log(1 + \max\{\varepsilon,\sqrt{\varepsilon}\}) + \delta + \varepsilon\cdot\max\{1,\delta\} = \varphi_2(\varepsilon).
\]
\end{proof}

\subsubsection{Uniform distance}

The following definition comes from \cite{DuKa}\footnote{More precisely, they define it without the logarithm which we add in order to satisfy the triangle inequality.}.

\begin{defin}[Uniform distance]\label{defin:UniformDist}Let $X$ and $Y$ be Banach spaces. If $u:X\to Y$  is uniformly continuous, we put
\[
	\Lip_{\infty}u:= \inf_{\eta>0}\;\sup\left\{\frac{\|u(x)-u(y)\|}{\|x-y\|}\setsep \|x-y\|\geq \eta\right\}.
\]
The \emph{uniform distance} between $X$ and $Y$ is defined as
\[
\rho_U(X,Y) = \inf \big\{\log ((\Lip_{\infty}u)(\Lip_{\infty}{u^{-1}}))\setsep u:X\to Y \text{ is uniform homeomorphism} \big\}.
\]
\end{defin}

Let us note the easy fact that we have
\[
\Lip_{\infty} u = \inf\big\{A>0 \setsep \exists B>0\; \forall x,y\in X:\; \|u(x)-u(y)\|\leq A\|x-y\| + B\big\}.
\]
The following is an analogue of Lemma~\ref{lem:LipschitzBySequenceOfCorr}.

\begin{lemma}\label{lem:UniformBySequenceOfCorr}
Let $\mu,\nu\in\Banach$. Then $\rho_U(X_{\mu},X_{\nu})<r$ if and only if there exist $B>0$, $r'\in (0,r)$ and a sequence of correspondences $\Corr_i\subseteq V\times V$ decreasing in inclusion such that
\begin{enumerate}
\item for every $i\in\Nat$ and every $v,w,v',w'\in V$ such that $v\Corr_i w$ and $v'\Corr_i w'$ we have $\nu(w-w')\leq \exp(r')\mu(v-v') + B$;
\item for $i\in\Nat$ and every $v,w,v',w'\in V$ such that $v\Corr_i w$ and $v'\Corr_i w'$ we have $\mu(v-v')\leq \nu(w-w') + B$;
\item for every $\varepsilon > 0$ there exist $\delta > 0$ and $i\in\Nat$ such that for every $v,v'\in V$ with $\mu(v-v')<\delta$ we have $\nu(w-w')<\varepsilon$ whenever $v\Corr_i w$ and $v'\Corr_i w'$;
\item for every $\varepsilon > 0$ there exist $\delta > 0$ and $i\in\Nat$ such that for every $w,w'\in V$ with $\nu(w-w')<\delta$ we have $\mu(v-v')<\varepsilon$ whenever $v\Corr_i w$ and $v'\Corr_i w'$.
\end{enumerate}
\end{lemma}
\begin{proof}
For the implication from the right to the left, for every $n\in V$ we define $\phi(n)\in X_{\nu}$ and $\psi(n)\in X_{\mu}$ as the unique element of $\bigcap_i \overline{n\Corr_i}$ and $\bigcap_i \overline{\Corr_i n}$, respectively. We leave to the reader to verify the simple fact that $\phi:(V,\mu)\rightarrow X_{\nu}$ is a uniformly continuous map with $\Lip_{\infty}\phi\leq\exp(r')$, which therefore extends to a uniformly continuous map $\bar \phi: X_\mu\rightarrow X_\nu$ and if $\bar \psi$ is defined analogously, then $\bar \phi=(\bar \psi)^{-1}$ and $\Lip_{\infty} \phi \Lip_{\infty} \psi < \exp(r)$.

For the other implication, suppose that we are given a uniform homeomorphism $u: X_{\mu}\rightarrow X_{\nu}$ such that $\Lip_{\infty} u^{-1} = 1$ and $\Lip_{\infty} u < \exp(r')$ for some $r' < r$. For every $i\in\Nat$ define correspondence $\Corr_i$ by
\[
\Corr_i:=\big\{(v,w)\in V\times V\setsep \exists \tilde{v}\in V\quad \mu(v-\tilde{v}) < \tfrac{1}{i}\;\; \&\;\; \nu(u(\tilde{v}) - w) < \tfrac{1}{i}\big\}.
\]
It is straightforward to check that $\Corr_i\subseteq \Nat\times\Nat$ are correspondences decreasing in inclusion satisfying all the conditions from the lemma. We omit further details, because this is similar to the proof of Lemma~\ref{lem:LipschitzBySequenceOfCorr}.
\end{proof}

\begin{remark}
As pointed out to us by the referee, it is natural to consider a distance $\rho'$ which would be defined as $\rho_U$ using $\Lip_\infty f$, where $f$ is a mapping which is not necessarily a uniform homeomorphism. Let us comment on this.

Recall that a mapping $f:X\to Y$ is a \emph{coarse Lipschitz map} if $\Lip_\infty f < \infty$, and that spaces $X$ and $Y$ are said to be \emph{coarse Lipschitz equivalent} (or sometimes also \emph{quasi-isometric}) if there are coarse Lipschitz maps $f:X\to Y$ and $g:Y\to X$ with $\sup_{y\in Y} \|f(g(y)) - y\| <\infty$ and $\sup_{x\in X} \|g(f(x)) - x\| < \infty$, in which case we say that $f$ and $g$ are \emph{coarsely inverse} to each other. It is well-known that two Banach spaces are coarse Lipschitz equivalent if and only if they have bi-Lipschitz equivalent nets (it holds even for general metric spaces, see e.g.\cite[Proposition 8.3.4]{BBI}). We leave to the reader the straightforward verification that for any Banach spaces $X$ and $Y$, $2\rho_N(X,Y)=\rho'(X,Y)$, where $\rho'$ is defined as $\inf\{\log((\Lip_\infty f)(\Lip_\infty g))\colon f: X\rightarrow Y, g: Y\rightarrow X\text{ are coarse Lipschitz maps coarsely inverse to each other}\}$, the factor $2$ is caused by different computation of distortion of Lipschitz maps, see Remark~\ref{rem:lip} for more details. Therefore we have covered this distance already above.
\end{remark}

\subsection{Analytic pseudometrics and reductions between them}

Here we recall the notions introduced in \cite{CDKpart1}, where we refer the interested reader for more information.

\begin{defin}
Let $X$ be a standard Borel space. A pseudometric $\rho:X\times X\to [0,\infty]$ is called \emph{an analytic pseudometric}, resp. \emph{a Borel pseudometric}, if for every $r > 0$ the set $\{(x,y)\in X^2\setsep \rho(x,y)<r\}$ is analytic, resp. Borel.

We emphasize that pseudometrics in our definition may attain $\infty$ as a value.
\end{defin}

Let us observe that distances considered in the previous subsection may be thought of as examples of analytic pseudometrics.

\begin{enumerate}[leftmargin=0cm,itemindent=.5cm,start=1,label={\bfseries \arabic*. }]
	\item {\bf Gromov-Hausdorff distance} Equip the Polish space $\Met$ with the Gromov-Hausdorff distance $\rho_{GH}$ defined in Definition \ref{defin:GHdistance}. We also consider the pseudometric $\rho_{GH}$ on the space $\Banach$ of codes for separable Banach spaces, denoted there by $\rho_{GH}^\Banach$. Note that for Banach spaces $X$ and $Y$, $\rho_{GH}^\Banach(X,Y)$ is defined as the Gromov-Hausdorff distance of the unit balls $B_X$ and $B_Y$ (see e.g. the introduction in \cite{KalOst}). Both $\rho_{GH}$ and $\rho_{GH}^\Banach$ are analytic, see \cite[Proposition 17]{CDKpart1}.

\smallskip

	\item {\bf Kadets distance} Equip the Polish space $\Banach$ with the Kadets distance $\rho_K$ defined in Definition \ref{defin:Kadetsdistance}. Using Lemma~\ref{lem:KadetsDescription}, it is not difficult to check that $\rho_K$ is analytic on $\Banach$. We leave the details to the reader.

\smallskip
    
    \item {\bf Lipschitz distance} Equip the Polish spaces $\Met$ and $\Banach$ with the Lipschitz distance $\rho_L$ introduced in Definition \ref{defin:LispchitzDistance}, where for $d,p\in\Met$ and $\mu,\nu\in\Banach$ by $\rho_L(d,p)$ and $\rho_L(\mu,\nu)$ we understand $\rho_L(M_d,M_p)$ and $\rho_L(X_\mu,X_\nu)$, respectively. We leave it to the reader to verify, using Lemma \ref{lem:LipschitzBySequenceOfCorr}, that $\rho_L$ is analytic on $\Met$ as well as on $\Banach$. Whenever we consider the pseudometric $\rho_L$ on $\Banach$ and we want to emphasize it, we write $\rho_L^\Banach$ instead of just $\rho_L$.

\smallskip

 \item {\bf Banach-Mazur distance} Equip the Polish space $\Banach$ by the Banach-Mazur distance $\rho_{BM}$ defined in Definition \ref{defin:BanachMazur}. We leave it to the reader to verify, using Lemma \ref{lem:BanachMazurVerifiableOnCountableSubspaces}, that $\rho_{BM}$ is an analytic pseudometric on $\Banach$.

\smallskip

\item {\bf Hausdorff-Lipschitz and net distances} Equip the Polish spaces $\Met$ and $\Banach$ with the Hausdorff-Lipschitz distance $\rho_{HL}$ from Definition \ref{defin:HausdorffLipschitz}. It is easy to check that for $d,p\in\Met$ we then have $$\rho_{HL}(d,p)=\inf\{\rho_{GH}(d,e_1)+\rho_L(e_1,e_2)+\rho_{GH}(e_2,p)\setsep e_1,e_2\in\Met\}.$$ Analogously, for elements from $\Banach$. It therefore follows from the fact that $\rho_{GH}$ and $\rho_L$ are analytic that $\rho_{HL}$ is analytic as well.

Moreover, equip the Polish space $\Banach$ with the net distance $\rho_N$ from Definition \ref{defin:netDistance}. It is clearly analytic as it coincides there with $\rho_{HL}$.

\smallskip

\item {\bf Uniform distance} Equip the Polish space $\Banach$ with the uniform distance $\rho_{U}$ from Definition \ref{defin:UniformDist}. We leave it to the reader to verify, using Lemma~\ref{lem:UniformBySequenceOfCorr}, that $\rho_U$ is an analytic pseudometric on $\Banach$.

\end{enumerate}

Now we recall the notion of reducibility between analytic pseudometrics as was introduced in \cite{CDKpart1}.

\begin{defin}\label{defin:Borel-UnifRedukce}
Let $X$, resp. $Y$ be standard Borel spaces and let $\rho_X$, resp. $\rho_Y$ be analytic pseudometrics on $X$, resp. on $Y$. We say that $\rho_X$ is \emph{Borel-uniformly continuous reducible} to $\rho_Y$, $\rho_X\leq_{B,u} \rho_Y$ in symbols, if there exists a Borel function $f: X\rightarrow Y$ such that, for every $\varepsilon>0$ there are $\delta_X>0$ and $\delta_Y>0$ satisfying
\[
\forall x,y\in X:\quad \rho_X(x,y)<\delta_X\Rightarrow \rho_Y(f(x),f(y))<\varepsilon
\]
and
\[
\forall x,y\in X:\quad \rho_Y(f(x),f(y))<\delta_Y\Rightarrow \rho_X(x,y)<\varepsilon.
\]
In this case we say that $f$ is a \emph{Borel-uniformly continuous reduction}. If $\rho_X\leq_{B,u} \rho_Y$ and $\rho_Y\leq_{B,u} \rho_X$, we say that $\rho_X$ is \emph{Borel-uniformly continuous bi-reducible} with $\rho_Y$ and write $\rho_X\sim_{B,u}\rho_Y$.

Moreover, if $f$ is injective we say it is an \emph{injective Borel-uniformly continuous reduction}.

If $f$ is an isometry from the pseudometric space $(X,\rho_X)$ into $(Y,\rho_Y)$, we say it is a \emph{Borel-isometric reduction}.

If there is $C>0$ such that for every $x,y\in X$ we have
\[
	\rho_Y(f(x),f(y))\leq C\rho_X(x,y)\quad\text{ and }\quad \rho_X(x,y)\leq C\rho_Y(f(x),f(y)),
\]
we say that $f$ is a \emph{Borel-Lipschitz reduction}.

If there are $\varepsilon > 0$ and $C>0$ such that for every $x,y\in X$ we have
\[\begin{split}
	\rho_X(x,y) < \varepsilon & \implies \rho_Y(f(x),f(y))\leq C\rho_X(x,y)\\
    \text{and}\quad \rho_Y(f(x),f(y)) < \varepsilon & \implies \rho_X(x,y)\leq C\rho_Y(f(x),f(y)),
\end{split}\]
we say that $f$ is a \emph{Borel-Lipschitz on small distances reduction}.
\end{defin}

The definition of a Borel-uniformly continuous reduction seems to be the most useful one. Sometimes we are able to demonstrate the reducibility between some pseudometrics by maps with stronger properties and this is the reason why we mentioned the remaining notions above.

\section{Reductions between pseudometrics on spaces of metric spaces}\label{sectionReduction1}

In this section we prove the reducibility results between pseudometrics on the spaces of metric spaces.

By \cite[Theorem 11]{CDKpart1}, we have the following.

\begin{thm}\label{thm:gh}Let $0<p<q$.
The following pseudometrics are mutually Borel-uniformly continuous bi-reducible: $\rho_{GH}$, $\rho_{GH}\upharpoonright \Met_p$, $\rho_{GH}\upharpoonright \Met^q$,  $\rho_{GH}\upharpoonright \Met_p^q$.
\end{thm}

Next, we show that Lipschitz and Gromov-Hausdorff distances are Borel-uniformly continuous bi-reducible.

\begin{thm}\label{thm:ReductionGHboundedToLip}
Fix real numbers $p<q$. Then the identity map on $\Met_p^q$ is a Borel-uniformly continuous reduction from $\rho_{GH}$ to $\rho_L$ and from $\rho_L$ to $\rho_{GH}$.

Moreover, the identity is not only Borel-uniformly continuous, but also Borel-Lipschitz on small distances.
\end{thm}
\begin{proof}
Take some $d,e\in \Met_p^q$ and suppose that $\rho_{GH}(d,e)<\varepsilon < p/2$. By Lemma \ref{lem:GHequivalenceUniformlyDiscrete} there is a permutation $\pi\in S_\infty$ witnessing that $d\simeq_{2\varepsilon} e$. The permutation $\pi$ also defines a bi-Lipschitz map between $(\Nat,d)$ and $(\Nat,e)$. Let us compute the Lipschitz constant of $\pi$. We have
\[
\mathrm{Lip}(\pi)=\sup_{m\neq n\in\Nat} \frac{e(\pi(m),\pi(n))}{d(m,n)}\leq \sup_{m\neq n\in\Nat} \frac{d(m,n)+2\varepsilon}{d(m,n)}\leq 1 + \frac{2\varepsilon}{p}.
\]
The same argument shows that $\mathrm{Lip}(\pi^{-1})\leq 1 + \frac{2\varepsilon}{p}$ and so we have $\rho_L(d,e)\leq \log(1+\tfrac{2\varepsilon}{p})\leq \tfrac{2\varepsilon}{p}$.

Conversely, suppose that $\rho_L(d,e)<\varepsilon <1$. Then there is a bi-Lipschitz map $\pi:(\Nat,d)\to (\Nat,e)$ such that $\max\{\mathrm{Lip}(\pi),\mathrm{Lip}(\pi^{-1})\}\leq 1+\delta$, where $1+\delta<\exp(\varepsilon)$. So for any $m,n\in\Nat$ we have 
\[\begin{split}
|d(m,n)-e(\pi(m),\pi(n))| \leq \max\big\{|(1+\delta)e(\pi(m),\pi(n))-e(\pi(m),\pi(n))|,\\
|(1+\delta)d(m,n)-d(m,n)|\big\} = \max\big\{\delta e(\pi(m),\pi(n)),\delta d(m,n)\big\}\leq \delta q.
\end{split}
\]
Thus, we have $d\simeq_{\delta q} e$ and, by Lemma~\ref{lem:lehciImplikace}, $\rho_{GH}(d,e)\leq \tfrac{\delta q}{2} < \tfrac{q(\exp(\varepsilon)-1)}{2}\leq \tfrac{q(\exp(1)-1)}{2}\varepsilon$.
\end{proof}

\begin{thm}\label{thm:ReductionLipToGH}
There is an injective Borel-uniformly continuous reduction from $\rho_{L}$ on $\Met$ to $\rho_{GH}$ on $\Met$.

Moreover, the reduction is not only Borel-uniformly continuous, but also Borel-Lipschitz on small distances.
\end{thm}

\begin{proof}
For every $ d \in \Met $, we define a metric $ \tilde{d} $ on $ (\mathbb{N} \times \mathbb{Z}) \cup \{ \clubsuit \} $ by
$$ \tilde{d}\big( (i,k), (j,l) \big) = |10k - 10l| + \min \big\{ 1, 2^{\min\{k,l\}}d(i,j) \big\}, $$
$$ \tilde{d}\big( (i,k), \clubsuit \big) = |10k + 4| + 1. $$
We leave it to the reader to verify the elementary fact that $\tilde{d}$ is a metric.

Since $ d \mapsto \tilde{d} $ is an injective continuous mapping from $ \Met $ into $ \mathbb{R}^{((\mathbb{N} \times \mathbb{Z}) \cup \{ \clubsuit \})^{2}} $, it is easy to show that there is an injective continuous mapping $ f : \Met \to \Met $ such that $ M_{f(d)} $ is isometric to the completion of $ ((\mathbb{N} \times \mathbb{Z}) \cup \{ \clubsuit \}, \tilde{d}) $. Hence, to prove the theorem, it is sufficient to show that
$$ \rho_{GH}(\tilde{d}, \tilde{e}) \leq (\exp \rho_{L}(d, e)) - 1 $$
and
$$ \rho_{GH}(\tilde{d}, \tilde{e}) < 1/4 \quad \Rightarrow \quad \rho_{L}(d, e) \leq \log (1 + 24 \rho_{GH}(\tilde{d}, \tilde{e})) $$
for every $ d, e \in \Met $.

Assume that $ (\exp \rho_{L}(d, e)) - 1 < \varepsilon $ and pick some $ L : M_{d} \to M_{e} $ with $ \Lip L < 1 + \varepsilon $ and $ \Lip L^{-1} < 1 + \varepsilon $. We define a correspondence
$$ \Corr = \{ (\clubsuit, \clubsuit) \} \cup \big\{ \big( (i, k), (j, k) \big) : d(i, L^{-1}(j)) < 2^{-k-1} \varepsilon, e(L(i), j) < 2^{-k-1} \varepsilon \big\}. $$
Our aim is to show that $ | \tilde{d}(a, a') - \tilde{e}(b, b') | < 2 \varepsilon $ whenever $ a \Corr b $ and $ a' \Corr b' $. In the case that $ a = a' = \clubsuit $ (and so $ b = b' = \clubsuit $), this is trivial. Assume that $ a \neq \clubsuit = a' $ (and so $ b \neq \clubsuit = b' $), and denote $ a = (i, k), b = (j, k) $. We want to show that $ | |10k + 4| + 1 - |10k + 4| - 1 | < 2 \varepsilon $, which is obvious. We obtain the same conclusion in the case $ a = \clubsuit \neq a' $. So, assume that $ a \neq \clubsuit \neq a' $ (and so $ b \neq \clubsuit \neq b' $), and denote $ a = (i, k), b = (j, k), a' = (i', k'), b' = (j', k') $. We want to show that $ | |10k - 10k'| + \min \{ 1, 2^{\min\{k,k'\}}d(i,i') \} - |10k - 10k'| - \min \{ 1, 2^{\min\{k,k'\}}e(j,j') \} | < 2 \varepsilon $, which can be slightly simplified to
$$ \big| \min \big\{ 1, 2^{\min\{k,k'\}}d(i,i') \big\} - \min \big\{ 1, 2^{\min\{k,k'\}}e(j,j') \big\} \big| < 2 \varepsilon. $$
Due to the symmetry, it is sufficient to show that the number under the absolute value is less than $ 2 \varepsilon $. If $ 1 \leq 2^{\min\{k,k'\}}e(j,j') $, then we just write $ \min \{ 1, 2^{\min\{k,k'\}}d(i,i') \} - 1 \leq 0 < 2 \varepsilon $. In the opposite case that $ 1 > 2^{\min\{k,k'\}}e(j,j') $, we write
\begin{align*}
d(i, i') & \leq d(L^{-1}(j), L^{-1}(j')) + d(i, L^{-1}(j)) + d(i', L^{-1}(j')) \\
 & < (1+\varepsilon) e(j, j') + 2^{-k-1} \varepsilon + 2^{-k'-1} \varepsilon \\
 & \leq (1+\varepsilon) e(j, j') + 2^{-\min\{k,k'\}-1} \varepsilon + 2^{-\min\{k,k'\}-1} \varepsilon
\end{align*}
and
\begin{align*}
\min \big\{ 1 & , 2^{\min\{k,k'\}}d(i,i') \big\} - \min \big\{ 1, 2^{\min\{k,k'\}}e(j,j') \big\} \\
 & \leq 2^{\min\{k,k'\}}d(i,i') - 2^{\min\{k,k'\}}e(j,j') \\
 & < 2^{\min\{k,k'\}} (1+\varepsilon) e(j, j') + 2^{-1} \varepsilon + 2^{-1} \varepsilon - 2^{\min\{k,k'\}}e(j,j') \\
 & = 2^{\min\{k,k'\}} \varepsilon \cdot e(j, j') + \varepsilon < \varepsilon + \varepsilon = 2\varepsilon.
\end{align*}
By Fact~\ref{fact:GHbyCorrespondences}, this completes the verification of $ \rho_{GH}(\tilde{d}, \tilde{e}) \leq \varepsilon $.

Now, assume that $ \rho_{GH}(\tilde{d}, \tilde{e}) < \varepsilon < 1/4 $ for some $ d, e \in \Met $. This is witnessed by a correspondence $ \Corr \subseteq [(\mathbb{N} \times \mathbb{Z}) \cup \{ \clubsuit \}]^{2} $ provided by Fact~\ref{fact:GHbyCorrespondences}. We verify first that $ \clubsuit \Corr \clubsuit $, showing that there is no $ m $ with $ m \Corr \clubsuit $ but $ \clubsuit $. If $ m \Corr \clubsuit $, then there are points which have distance to $ m $ in $ (5-2\varepsilon, 5+2\varepsilon) $ and in $ (7-2\varepsilon, 7+2\varepsilon) $. If two points have distance in $ (3, 7) $, then these points belong to $ (\mathbb{N} \times \{ 0 \}) \cup \{ \clubsuit \} $. As $ (5-2\varepsilon, 5+2\varepsilon) \subseteq (3, 7) $, we obtain $ m \in (\mathbb{N} \times \{ 0 \}) \cup \{ \clubsuit \} $. The case $ m \in \mathbb{N} \times \{ 0 \} $ is also excluded, since distances of these points to other points belong to the set $ [0, 1] \cup \{ 5 \} \cup [10, \infty) $, which is disjoint from $ (7-2\varepsilon, 7+2\varepsilon) $.

It follows that
\begin{equation} \label{liptogh1}
(i,k) \Corr (j,l) \quad \Rightarrow \quad k = l.
\end{equation}
Indeed, we have $ 2 > 2\varepsilon > | \tilde{d}((i,k), \clubsuit) - \tilde{e}((j,l), \clubsuit) | = ||10k+4| - |10l+4|| $, and this is possible only if $ k = l $.

By \eqref{liptogh1}, the relations
$$ \Corr_{k} = \big\{ (i, j) \in \mathbb{N}^{2} : (i,k) \Corr (j,k) \big\}, \quad k \in \mathbb{Z}, $$
are correspondences. Let us show that
\begin{equation} \label{liptogh2}
i \Corr_{k} j, \; i' \Corr_{k} j' \; \& \; d(i, i') \leq 2^{-k-1} \quad \Rightarrow \quad e(j, j') \leq d(i, i') + 2^{1-k} \varepsilon.
\end{equation}
Using $ \tilde{e}((j,k), (j',k)) < \tilde{d}((i,k), (i',k)) + 2\varepsilon $, we obtain
$$ \min \big\{ 1, 2^{k}e(j, j') \big\} < \min \big\{ 1, 2^{k}d(i, i') \big\} + 2\varepsilon \leq 2^{k}d(i, i') + 2\varepsilon. $$
In particular, using $ d(i, i') \leq 2^{-k-1} $, the minimum on the left hand side is less than $ 1/2 + 2\varepsilon $, hence less than $ 1 $ and equal to $ 2^{k}e(j, j') $. Thus, \eqref{liptogh2} follows.

Further, we show that
\begin{equation} \label{liptogh3}
i \Corr_{k} j \; \& \; i \Corr_{k'} j' \quad \Rightarrow \quad e(j, j') \leq 2^{1-\min\{k,k'\}} \varepsilon.
\end{equation}
Using $ | \tilde{e}((j,k), (j',k')) - \tilde{d}((i,k), (i,k')) | < 2\varepsilon $, we obtain $ | |10k - 10k'| + \min \{ 1, 2^{\min\{k,k'\}}e(j, j') \} - |10k - 10k'| | < 2\varepsilon $. That is,
$$ \min \{ 1, 2^{\min\{k,k'\}}e(j, j') \} < 2\varepsilon, $$
which gives \eqref{liptogh3}.

Let us consider the decreasing sequence of correspondences given by
$$ \Corr^{*}_{s} = \bigcup_{k=s}^{\infty} \Corr_{k} = \big\{ (i, j) \in \mathbb{N}^{2} : (\exists k \geq s)\big( (i,k) \Corr (j,k) \big) \big\}, \quad s = 1, 2, \dots, $$
and let us observe that
\begin{equation} \label{liptogh4}
e\text{-}\diam (i \Corr^{*}_{s}) \leq 2^{1-s} \varepsilon, \quad d\text{-}\diam (\Corr^{*}_{s} j) \leq 2^{1-s} \varepsilon.
\end{equation}
The first inequality follows from \eqref{liptogh3}, the second one holds due to the symmetry.

We claim moreover that
\begin{equation} \label{liptogh5}
i \Corr^{*}_{s} j, \; i' \Corr^{*}_{s} j' \; \& \; d(i, i') \geq 2^{-s-2} \quad \Rightarrow \quad e(j, j') \leq (1+24\varepsilon) d(i, i').
\end{equation}
Let $ k \in \mathbb{Z} $ be such that $ 2^{-k-2} \leq d(i, i') < 2^{-k-1} $. We have $ k \leq s $, and so $ i \Corr^{*}_{k} j, i' \Corr^{*}_{k} j' $ in particular. Let us pick $ n, n' \in \mathbb{N} $ such that $ i \Corr_{k} n, i' \Corr_{k} n' $. We obtain from \eqref{liptogh2} that $ e(n, n') \leq d(i, i') + 2^{1-k} \varepsilon $. Also, we obtain from \eqref{liptogh4} that $ e(j, n) \leq 2^{1-k} \varepsilon $ and $ e(j', n') \leq 2^{1-k} \varepsilon $. By the triangle inequality,
$$ e(j, j') \leq d(i, i') + 2^{1-k} \varepsilon + 2^{1-k} \varepsilon + 2^{1-k} \varepsilon = d(i, i') + 24 \cdot 2^{-k-2} \varepsilon \leq (1+24\varepsilon) d(i, i'), $$
which gives \eqref{liptogh5}. Let us note that, due to the symmetry, we have also
\begin{equation} \label{liptogh6}
i \Corr^{*}_{s} j, \; i' \Corr^{*}_{s} j' \; \& \; e(j, j') \geq 2^{-s-2} \quad \Rightarrow \quad d(i, i') \leq (1+24\varepsilon) e(j, j').
\end{equation}

Finally, applying Lemma~\ref{lem:LipschitzBySequenceOfCorr} and using \eqref{liptogh4}, \eqref{liptogh5} and \eqref{liptogh6}, we obtain $ \rho_{L}(d, e) \leq \log (1+24\varepsilon) $.
\end{proof}

\begin{thm}\label{thm:GhbiReducibleWithLipUniDiscrete}
The pseudometrics $\rho_{GH}$ on $\Met$ and $\rho_L$ on $\Met$ are Borel-uniformly continuous bi-reducible.
\end{thm}
\begin{proof}
By Theorem~\ref{thm:gh} and Theorem \ref{thm:ReductionGHboundedToLip} we get
\[
\rho_{GH}\; \leq_{B,u}\; (\rho_{GH}\upharpoonright \Met_2^4)\; \leq_{B,u}\; (\rho_{L}\upharpoonright \Met_2^4)\; \leq_{B,u}\;\rho_{L}.
\]

For the other direction, we use Theorem \ref{thm:ReductionLipToGH}.
\end{proof}

Finally, using an analogous proof, we obtain that the Hausdorff-Lipschitz distance on metric spaces is reducible to the Gromov-Hausdorff distance. We will see later it is actually bi-reducible with it (Theorem~\ref{thmGHtoK}).

\begin{thm}\label{thm:ReductionHLToGH}
There is an injective Borel-uniformly continuous reduction from $\rho_{HL}$ on $\Met$ to $\rho_{GH}$ on $\Met$.
\end{thm}
\begin{proof}Denote by $\Nat^-$ the set $\{k\in\mathbb{Z}\setsep k\leq 0\}$. For every $ d \in \Met $, we define a metric $ \tilde{d} $ on $ (\mathbb{N} \times \Nat^-) \cup \{ \clubsuit \} $ by
$$ \tilde{d}\big( (i,k), (j,l) \big) = |10k - 10l| + \min \big\{ 1, 2^{\min\{k,l\}}d(i,j) \big\}, $$
$$ \tilde{d}\big( (i,k), \clubsuit \big) = |10k + 4| + 1. $$
Note that this is the same construction which we used already in the proof of Theorem~\ref{thm:ReductionLipToGH} with the exception that the underlying set is $ (\Nat \times \Nat^-) \cup \{ \clubsuit \} $ and in the proof of Theorem~\ref{thm:ReductionLipToGH} it is $ (\Nat \times \mathbb{Z}) \cup \{ \clubsuit \} $. Hence, to prove the theorem, it is sufficient to show that for every $\varepsilon > 0$ there are $\delta_1 > 0$ and $\delta_2 > 0$ such that 
$$ \rho_{GH}(\tilde{d}, \tilde{e}) < \delta_1 \Rightarrow \rho_{HL}(d, e) \leq \varepsilon$$
and
$$ \rho_{HL}(d, e) < \delta_2 \Rightarrow \rho_{GH}(\tilde{d}, \tilde{e}) \leq \varepsilon$$
for every $ d, e \in \Met $.

By Lemma~\ref{lem:HLByCorrespondences}, there exists $\delta' > 0$ such that
\[\begin{split}
\text{$d$ and $e$ are $HL(\delta')$-close}\quad & \Rightarrow \quad \rho_{HL}(d, e) < \varepsilon,\\
\rho_{HL}(d, e) < \delta' \quad & \Rightarrow \quad \text{$d$ and $e$ are $HL(\varepsilon)$-close}.
\end{split}\]
We claim that it suffices to put $\delta_1 = \min\{\tfrac{1}{5},\tfrac{\delta'}{24}\}$ and $\delta_2 = \delta'$.

Assume that $\rho_{GH}(\tilde{d}, \tilde{e}) < \delta_1$. This is witnessed by a correspondence $\Corr\subseteq [(\Nat \times \Nat^-) \cup \{ \clubsuit \}]^2$. Then, using verbatim the same arguments as in the proof of Theorem~\ref{thm:ReductionLipToGH}, the relation $\Corr_0 = \{(i,j)\in\Nat^2\setsep (i,0)\Corr(j,0)\}$ is a correspondence and whenever $i\Corr_0 j$ and $i'\Corr_0 j'$, we have
\[\begin{split}
	d(i,i')\leq \tfrac{1}{2}\quad & \Rightarrow \quad e(j,j')\leq d(i,i') + 2\delta_1,\\
	d(i,i')\geq \tfrac{1}{4}\quad & \Rightarrow \quad e(j,j')\leq (1+24\delta_1)d(i,i');    
\end{split}\]
and similarly for the symmetric situation when the roles of $d$ and $e$ are changed. In particular, $\Corr_0$ witnesses the fact that $d$ and $e$ are $HL(24\delta_1)$-close and since $24\delta_1 \leq \delta'$, we have $\rho_{HL}(d,e) < \varepsilon$.

Assume that $\rho_{HL}(d,e) < \delta_2$. Then $d$ and $e$ are $HL(\varepsilon)$-close, which is witnessed by a correspondence $\Corr'\subseteq \Nat^2$. Similarly as in the proof of Theorem~\ref{thm:ReductionLipToGH}, we define a correspondence
\[
	\Corr = \{(\clubsuit,\clubsuit)\}\cup \big\{((i,k),(j,k))\setsep i\Corr'j,\; k\leq 0\big\}.
\]
Our aim is to show that $|\tilde{d}(a,a') - \tilde{e}(b,b')| < 2\varepsilon$ whenever $a\Corr b$ and $a'\Corr b'$. Using verbatim the same arguments as in the proof of Theorem~\ref{thm:ReductionLipToGH}, it is sufficient to show that for $a=(i,k)$, $b=(j,k)$, $a'=(i',k')$, $b'=(j',k')$ with $1 > 2^{\min\{k,k'\}}e(j,j')$, $i\Corr'j$, $i'\Corr'j'$ we have
\begin{equation}\label{eq:HlWanted}
	2^{\min\{k,k'\}}d(i,i') - 2^{\min\{k,k'\}}e(j,j') < 2\varepsilon.
\end{equation}
Fix $a,a',b,b'$ as above. If $e(j,j')\leq 1$, using that $\Corr'$ witnesses $d$ and $e$ are $HL(\varepsilon)$-close, we get $2^{\min\{k,k'\}}(d(i,i') - e(j,j')) < 2\varepsilon$. On the other hand, if $e(j,j')\geq 1$, we get $2^{\min\{k,k'\}}(d(i,i') - e(j,j'))\leq 2^{\min\{k,k'\}}((1+\varepsilon)e(j,j') - e(j,j')) = 2^{\min\{k,k'\}}\varepsilon e(j,j') < \varepsilon$. Hence, \eqref{eq:HlWanted} holds and so the correspondence $\Corr$ witnesses that $\rho_{GH}(\tilde{d},\tilde{e})\leq \varepsilon$.
\end{proof}

\section{Reductions from pseudometrics on \texorpdfstring{$\Banach$}{B} to pseudometrics on \texorpdfstring{$\Met$}{M}}\label{sectionReduction2}
We start with a reduction from the Banach-Mazur distance to the Lipschitz distance. An essential ingredient is Lemma~\ref{lem:BanachMazurVerifiableOnCountableSubspaces}.
\begin{thm}\label{thm:BMtoGH}
There is a Borel-uniformly continuous reduction from $\rho_{BM}$ to $\rho_L$ on $\Met_p^q$, where $0<p<q$.

Moreover, the reduction is not only Borel-uniformly continuous, but also Borel-Lipschitz on small distances.
\end{thm}
\begin{proof}
Without loss of generality, we assume that $p=2$ and $q=15$. The structure of the proof is the following. First, we describe a construction which to each $\nu\in\Banach$ assigns
a metric space $M_\nu$. Next, we show that for $\nu,\lambda\in\Banach$ we have $\rho_L(M_\nu,M_\lambda)\leq \rho_{BM}(\nu,\lambda)$ and
\[
\rho_L(M_\nu,M_\lambda) < \log(\tfrac{4}{3}) \implies \rho_{BM}(\nu,\lambda)\leq 2\rho_L(M_\nu,M_\lambda).
\]
Finally, we show it is possible to make such an assignment in a Borel way.

Fix some countable sequence $(c_i)_{i\in\Nat}$ of positive real numbers such that for every positive real number $r>0$ there exists $i\geq 7$ such that $c_i\cdot r\in (2, 9/4)$. Also, let $\pi:\Rat\rightarrow \Nat\setminus\{1\}$ be some bijection and let $\preceq$ be some linear order on $V$. To each $\nu\in\Banach$ we assign a countable metric space $M_\nu$ with the following underlying set: $$V\;\cup\;\big\{p_{a,b}^{m,k}\setsep a\preceq b\in V, m\geq 7, k\leq m\big\}\;\cup$$ $$ \cup \; \big\{f_{a,q}^m\setsep a\in V,q\in \Rat,m\leq \pi(q)\big\}\;\cup \;\big\{x_{a,b}^i\setsep a\preceq b\in V, i\leq 3\big\}.$$

\medskip

The metric $d_\nu$ on $M_\nu$ is defined as follows.
\begin{itemize}
\item For every $a\neq b\in V$ we set $d_\nu(a,b)=15$.
\item For every $a\preceq b\in V$, $m\geq 7$ we define the number $K^m_{a,b}$ to be $\max\{2,\min\{3, c_m\cdot \nu(a-b)\}\}$. Then we set $d_\nu(a,p^{m,1}_{a,b})=d_\nu(p^{m,1}_{a,b},p^{m,2}_{a,b})=\ldots=d_\nu(p^{m,m}_{a,b},b)=K^m_{a,b}$.
\item For every $a\in V$ and $q\in\Rat$ we set $d_\nu(a,f_{a,q}^1)=7$, $d_\nu(f_{a,q}^1,f_{a,q}^2)=\ldots=d_\nu(f_{a,q}^{\pi(q)},qa)=10$.
\item For every $a\preceq b\in V$ we set $d_\nu(a,x_{a,b}^1)=d_\nu(b,x_{a,b}^2)=d_\nu(x_{a,b}^1,x_{a,b}^3)=d_\nu(x_{a,b}^2,x_{a,b}^3)=d_\nu(x_{a,b}^3,a+b)=5$.
\item On the rest of $M_\nu^2$, we take the greatest extension of $d_\nu$ defined above with $15$ as the upper bound, which is nothing but the graph metric (bounded by $15$).
\end{itemize}

We shall call the pairs of elements from $M_\nu$, for which the distance was defined directly before taking the extension, \emph{edges}. In order to simplify some notation, whenever we write $p_{b,a}^{m,k}$, where $a\preceq b$, we mean the element $p_{a,b}^{m,k}$. Also by $p_{a,b}^{m,0}$ we mean the element $a$, and by $p_{a,b}^{m,m+1}$ we mean the element $b$. We shall call the pairs $p^{m,k}_{a,b}, p^{m,k+1}_{a,b}$ \emph{neighbors}.

\medskip
Consider two norms $\nu,\lambda\in\Banach$. Denote the elements of $M_\lambda$ by $M_\lambda = $ $V\;\cup\;\{q_{a,b}^{m,k}\setsep a\preceq b\in V, m\geq 7, k\leq m\}\;\cup\; \{g_{a,q}^m\setsep a\in V,q\in \Rat,m\leq \pi(q)\}\;\cup \;\{y_{a,b}^i\setsep a\preceq b\in V, i\leq 3\}$ and the numbers $\max\{2,\min\{3, c_m\cdot \lambda(a-b)\}\}$ by $L^m_{a,b}$.

We claim that $\rho_L(M_\nu,M_\lambda)\leq \rho_{BM}(\nu,\lambda)$. If $\rho_{BM}(\nu,\lambda)<\varepsilon$, by Lemma \ref{lem:BanachMazurVerifiableOnCountableSubspaces}, there exists a surjective $\Rat$-linear isomorphism $T:(V,\nu)\rightarrow (V,\lambda)$ with $\|T\|\|T^{-1}\|<\exp(\varepsilon)$. Fix $\varepsilon'>0$. We may assume that $\min\{\|T\|,\|T^{-1}\|\}\geq 1-\varepsilon'$. We use $T$ to define a bi-Lipschitz bijection $T':M_\nu\rightarrow M_\lambda$. For every $a\in V$ we set $T'(a)=T(a)$ and for all elements of the form $p_{a,b}^{m,k}$, $f_{a,q}^m$, and $x_{a,b}^i$, with appropriate indices, whenever $T(a)\preceq T(b)$, we set $T'(p_{a,b}^{m,k})=q_{T(a),T(b)}^{m,k}$, $T'(f_{a,q}^m)=g_{T(a),q}^m$, and $T'(x_{a,b}^i)=y_{T(a),T(b)}^i$; similarly, if $T(b)\prec T(a)$, we set $T'(p_{a,b}^{m,k})=q_{T(b),T(a)}^{m+1-k,k}$, $T'(f_{a,q}^m)=g_{T(a),q}^m$, $T'(x_{a,b}^1)=y_{T(b),T(a)}^2$, $T'(x_{a,b}^2)=y_{T(b),T(a)}^1$ and $T'(x_{a,b}^3)=y_{T(b),T(a)}^3$.
Let us compute the Lipschitz constants of $T'$. If $a=b$, then obviously $\tfrac{L^m_{T(a),T(b)}}{K^m_{a,b}} = 1$. Otherwise, we have
\[\begin{split}
\frac{L^m_{T(a),T(b)}}{K^m_{a,b}} & =\frac{\max\{2,\min\{3, c_m\cdot \lambda(T(a)-T(b))\}\}}{\max\{2,\min\{3, c_m\cdot \nu(a-b)\}\}}\\
& \leq \max\left\{1,\frac{\lambda(T(a)-T(b))}{\nu(a-b)}\right\}\leq \max\{1,\|T\|\}\leq \|T\| + \varepsilon',
\end{split}\]
where in the first inequality we used the easy fact that for $x,y > 0$ we have $\frac{\max\{2,\min\{3,x\}\}}{\max\{2,\min\{3,y\}\}}\leq \max\{1,\frac{x}{y}\}$. It follows that $\Lip(T')\leq \|T\| + \varepsilon'$.  Indeed, it follows from the definition of $T'$ that it maps edges onto edges. Moreover, for every edge $(x,y)\in M_\nu^2$ we have $d_\lambda(T'(x),T'(y))\leq (\|T\| + \varepsilon')d_\nu(x,y)$, so the same inequality extends to the graph metrics -- the extensions of $d_\nu$ and $d_\lambda$ on the whole $M_\nu$ and $M_\lambda$ respectively. We obtain in particular that $ \Lip(T) \leq (\|T\| + \varepsilon') \| T^{-1} \|/(1 - \varepsilon') $. Since an analogous inequality holds for $ \Lip((T')^{-1}) $ and $ \varepsilon' > 0 $ was arbitrary, we have $ \rho_L(M_\nu,M_\lambda) \leq \log (\|T\| \|T^{-1}\|) < \varepsilon $. Thus, we conclude that $\rho_L(M_\nu,M_\lambda)\leq \rho_{BM}(\nu,\lambda)$.

\medskip
Conversely, assume that $\exp(\rho_L(M_\nu,M_\lambda))<4/3$, that is, there exists a bijection $T:M_\nu\rightarrow M_\lambda$ with $\Lip(T)<4/3$ and $\Lip(T^{-1})<4/3$. We will show that $\rho_{BM}(\nu,\lambda)\leq 2\rho_L(M_\nu,M_\lambda)$.

First we claim that $T$ maps $V\subseteq M_\nu$ bijectively onto $V\subseteq M_\lambda$. Indeed, the points $a\in V\subseteq M_\nu$ are characterized as those points $x$ of $M_\nu$ for which there exist infinitely many points $y\in M_\nu$ with $\nu(x-y)\leq 3$. On the other hand, the points from $M_\nu\setminus V$ are characterized as those points $x$ of $M_\nu$ for which there are at most two points distinct from $x$ of distance less than $4$ from $x$. Since $\Lip(T)<4/3$, we get $T(V)\subseteq V$ and similarly we have $T^{-1}(V)\subseteq V$, which proves the claim. We denote by $S$ the induced bijection between $(V,\nu)$ and $(V,\lambda)$.

We claim that $S$ is $\Rat$-linear. Let us check that it is homogeneous for all rationals, that is, $S(qa)=qS(a)$ for all $a\in V$ and $q\in\Rat$, which in particular gives that $S(0)=0$. For each $a\in V\subseteq M_\nu$ and $q\in\Rat$ there is a path of points $a,f_{a,q}^1,\ldots,f_{a,q}^{\pi(q)},qa$. The map $T$ must send this path to some path $T(a),g_{T(a),q'}^1,\ldots,g_{T(a),q'}^{\pi(q')}, q'T(a)$. However, $q'$ is determined by the length of the path which must be the same as the length of the former path. Therefore $q'=q$ and $S(qa)=T(qa)=qT(a)=qS(a)$. Next, we show that for $a\neq b\in V$ we have $S(a+b)=S(a)+S(b)$. There is a ``triangle of paths'' formed by the points $a,b,x_{a,b}^1,x_{a,b}^2,x_{a,b}^3,a+b$. $T$ must preserve this triangle, so it maps it to a triangle formed by the points $T(a),T(b),y_{T(a),T(b)}^1,y_{T(a),T(b)}^2,y_{T(a),T(b)}^3, T(a)+T(b)$. That shows that $S(a+b)=T(a+b)=T(a)+T(b)=S(a)+S(b)$.

It remains to compute the Lipschitz constant of $S$, resp. $S^{-1}$, as a map from $(V,\nu)$ to $(V,\lambda)$. In order to do it, we claim that for every $a\preceq b$, $ a \neq b $, $ m \geq 7 $ and $k\leq m$ we have $T(p^{m,k}_{a,b})=q^{m,k}_{S(a),S(b)}$ if $S(a)\preceq S(b)$, and $T(p^{m,k}_{a,b})=q^{m,m+1-k}_{S(b),S(a)}$ if $S(b)\preceq S(a)$. We only treat the former case, the other is treated analogously. First observe that $T(p^{m,1}_{a,b})=q^{m',k'}_{S(a),b'}$, for some $m'$ and $b'$, and $k'=1$ or $k'=m'$. Indeed, $p^{m,1}_{a,b}$ is a neighbor of $a$, so $d_\nu(a,p^{m,1}_{a,b})\leq 3$. Therefore $d_\lambda(S(a), T(p^{m,1}_{a,b}))<4$, so $S(a)$ and $T(p^{m,1}_{a,b})$ are also neighbors. Analogously, we show that for every $0\leq k\leq m$ we have that $T(p^{m,k}_{a,b})$ and $T(p^{m,k+1}_{a,b})$ are neighbors, which implies that $T$ indeed maps the `path' $a,p^{m,1}_{a,b},p^{m,2}_{a,b},\ldots,p^{m,m}_{a,b},b$ onto the path $S(a), q^{m,1}_{S(a),S(b)}, q^{m,2}_{S(a),S(b)},\ldots, q^{m,m}_{S(a),S(b)}, S(b)$.

We are now ready to compute the Lipschitz constants. We do it for $S$. Pick some $a\preceq b$, $ a \neq b $. We want to compute $\frac{\lambda(S(a)- S(b))}{\nu(a-b)}$. We consider only the case when $S(a)\preceq S(b)$, the other case is analogous. By the choice of $ (c_{i})_{i \in \mathbb{N}} $, there exists $ m \geq 7 $ such that $c_m\cdot \nu(a-b) \in (2,9/4)$. It follows that $d_\nu(a,p^{m,1}_{a,b})\in (2,9/4)$, so we have
\[
d_\lambda(S(a), q^{m,1}_{S(a),S(b)})=d_\lambda(T(a), T(p^{m,1}_{a,b}))\leq \Lip(T)d_\nu(a, p^{m,1}_{a,b})<3,
\]
which implies that
\[
\lambda(S(a) - S(b))\leq \tfrac{d_\lambda(S(a),q^{m,1}_{S(a),S(b)})}{c_m}\leq \tfrac{\Lip(T)d_\nu(a,p^{m,1}_{a,b})}{c_m}=\Lip(T)\nu(a - b).
\]
That shows that $\|S\| \leq \Lip(T)$. Analogously, we get $\|S^{-1}\| \leq \Lip(T^{-1})$; hence, we have $\rho_{BM}(\nu,\lambda)\leq 2\log \max \{ \Lip(T), \Lip(T^{-1}) \}$. Considering all bi-Lipschitz maps $T$ with $\Lip(T)<4/3$ and $\Lip(T^{-1})<4/3$, we obtain $\rho_{BM}(\nu,\lambda)\leq 2\rho_L(M_\nu,M_\lambda)$ whenever $\rho_L(M_\nu,M_\lambda) < \log(4/3)$.

\medskip
Finally, to verify that the map $\Banach \ni \nu\to (M_\nu,d_\nu)$ is Borel, let us denote by $N$ the underlying set of $M_\nu$ (which is the same for every $\nu\in\Banach$). Now  it suffices to fix some bijection $\phi:\mathbb{N} \to N$ and check that the distances in $M_\nu$ depend on distances of $\nu$ in a continuous (when considering $\nu$ as a member of $\Rea^V$) way.
\end{proof}
A consequence of the last theorem and Theorem~\ref{thm:ReductionGHboundedToLip} is that the Banach-Mazur distance is Borel-uniformly continuous reducible to the Gromov-Hausdorff distance. We will see later it is actually bi-reducible with it.
\begin{cor}
We have $\rho_{BM}\leq_{B,u} \rho_{GH}$.
\end{cor}

Next we show that the Hausdorff-Lipschitz distance on Banach spaces is reducible to the Gromov-Hausdorff distance. Again, we will see later it is
actually bi-reducible with it (Theorem~\ref{thmGHtoK}).

The reduction is obtainable already from Theorem~\ref{thm:ReductionHLToGH}. However, the proof which follows is in this concrete case more natural and gives a slightly better result, that is, the reduction is even Borel-Lipschitz on small distances. 
\begin{thm}\label{thm:ReductionHLdistToGH}
There is an injective Borel-uniformly continuous reduction from $\rho_{HL}$, equivalently $\rho_N$, on $\Banach$ to $\rho_{GH}$ on $\Met$.

Moreover, the reduction is not only Borel-uniformly continuous, but also Borel-Lipschitz on small distances. 
\end{thm}
\begin{proof}
To every separable Banach space $(X,\Norm_X)$ we associate a metric space $(X,d_X)$ whose underlying set is unchanged, and for every $x,y\in X$ we set $d_X(x,y)=\min\{\|x-y\|_X, 1\}$. We claim that the map $(X,\Norm_X)\to (X,d_X)$ is the desired reduction.

Fix some separable Banach spaces $X$ and $Y$. Suppose first that $\rho_{HL}(X,Y)=\rho_N(X,Y)<K$, for some $K>0$, where the first equality follows from Proposition \ref{prop:HLandNetDistanceAreEqual}. So there exist nets $\Net_X\subseteq X$ and $\Net_Y\subseteq Y$ and a bi-Lipschitz map $T:\Net_X\rightarrow \Net_Y$ with $\log\max\{\Lip(T),\Lip(T^{-1})\}<K$. Pick any $\varepsilon>0$. By rescaling the nets $\Net_X$ and $\Net_Y$ if necessary we may assume (as in the proof of Proposition \ref{prop:HLandNetDistanceAreEqual}) the nets are an $(a,\varepsilon)$-net, resp. an $(a',\varepsilon)$-net, for some $a,a'>0$. Since $(\Net_X,d_X)$ and $(\Net_Y,d_Y)$ belong to $\Met^1_{\min(a,a')}$ we get from Theorem \ref{thm:ReductionGHboundedToLip} that $\rho_{GH}((\Net_X,d_X),(\Net_Y,d_Y))\leq (\exp(K)-1)/2$. Since $\rho_{GH}((X,d_X),(\Net_X,d_X))\leq \varepsilon$, $\rho_{GH}((Y,d_Y),(\Net_Y,d_Y))\leq \varepsilon$, and since $\varepsilon$ was arbitrary, we get that $\rho_{GH}((X,d_X),(Y,d_Y))\leq\tfrac{\exp(K)-1}{2}\leq \tfrac{\exp(1)-1}{2}K$ whenever $K<1$.

\medskip
Conversely, suppose that $\rho_{GH}((X,d_X),(Y,d_Y))<K$, where $K<1/4$. By Remark~\ref{rem:GHPerfectSpaces} there exists a bijection $\phi:X\rightarrow Y$ witnessing the Gromov-Hausdorff distance, i.e. for every $x,y\in X$ we have $|d_X(x,y)-d_Y(\phi(x),\phi(y))|<2K$. We aim to show that $\phi$ is large scale bi-Lipschitz for $(X,\Norm_X)$ and $(Y,\Norm_Y)$. Pick any $x,y\in X$ with $\|x-y\|_X\geq 1$. Find points $x_0=x,x_1,x_2,\ldots,x_{n-1},x_n=y$ such that $\sum_{i=0}^{n-1}\|x_i-x_{i+1}\| = \|x-y\|$, $n\leq 3\|x-y\|_X$ and for every $i<n$ we have $\|x_i-x_{i+1}\|\leq 1/2$. Notice that for every $i<n$ we have $\|x_i-x_{i+1}\|_X=d_X(x_i,x_{i+1})\leq 1/2$. So $d_Y(\phi(x_i),\phi(x_{i+1}))\leq d_X(x_i,x_{i+1})+2K<1$, therefore $\|\phi(x_i)-\phi(x_{i+1})\|_Y=d_Y(\phi(x_i),\phi(x_{i+1}))\leq \|x_i-x_{i+1}\|_X+2K$.

Now we compute 
\[\begin{split}
\|\phi(x)-\phi(y)\|_Y & \leq \sum_{i=0}^{n-1} \|\phi(x_i)-\phi(x_{i+1})\|_Y\leq\|x-y\|_X+2K(3\|x-y\|_X)\\
& = (1+6K)\|x-y\|_X.
\end{split}\]
Along with the analogous computations for $\phi^{-1}$ we get that $\max\{\Lip_1(\phi),\Lip_1(\phi^{-1})\}\leq 1+6K$, where $$\Lip_1(\phi)=\sup_{\substack{x,y\in X\\ \|x-y\|_X\geq 1}}\frac{\|\phi(x)-\phi(y)\|_Y}{\|x-y\|_X}$$
and the analogous definition holds for $\phi^{-1}$.
Now it suffices to choose some maximal $2$-separated set $\Net_X$ in $(X,\Norm_X)$, which is a net in $X$. Its image $\phi[\Net_X]$, denoted by $\Net_Y$, is a net in $Y$. Indeed, we claim that for each $x\neq y\in \Net_X$ we have $\|\phi(x)-\phi(y)\|_Y> 1$. Otherwise, there is $z\in Y$ such that $\|\phi(x)-z\|_Y\leq 1/2$ and $\|z-\phi(y)\|_Y\leq 1/2$. This implies that $\|x-y\|_X\leq \|x-\phi^{-1}(z)\|_X+\|\phi^{-1}(z)-y\|_X<\|\phi(x)-z\|_Y+\|z-\phi(y)\|_Y+4K\leq 2$, a contradiction. Finally, we claim that for every $y\in Y$ we can find $y'\in \Net_Y$ with $\|y-y'\|_Y<4$. Pick any $y\in Y$. Then there exist $x_1,x_2,x_3\in X$ and $x'\in\Net_X$ such that $\max\{\|\phi^{-1}(y)-x_1\|_X,\|x_1-x_2\|_X,\|x_2-x_3\|_X,\|x_3-x'\|_X\}\leq 1/2$. Set $y'=\phi(x')$. We get that $\|y-y'\|_Y\leq \|y-\phi(x_1)\|_Y+\|\phi(x_1)-\phi(x_2)\|_Y+\|\phi(x_2)-\phi(x_3)\|_Y+\|\phi(x_3)-y'\|_Y\leq 2 +8K < 4$. So we have verified that $\Net_Y$ is a net. It is bi-Lipschitz with $\Net_X$ as witnessed by $\phi$. So we get the estimate $\rho_{HL}((X,\Norm_X),(Y,\Norm_Y))=\rho_N((X,\Norm_X),(Y,\Norm_Y))\leq \log (1+6K)\leq 6K$.

\medskip
Finally, we observe that the map $(X,\Norm_X)\to (X,d_X)$ can be viewed as a Borel function from $\Banach$ to $\Met$. Recall that elements of $\Banach$ are norms on a countable infinite-dimensional $\Rat$-vector space denoted by $V$. By fixing a bijection $f: V\rightarrow \Nat$ we associate to each $\Norm\in\Banach$ a metric $d\in\Met$ such that for every $n,m\in\Nat$ we have $d(n,m)=\min\{1,\|f^{-1}(n)-f^{-1}(m)\|\}$. This is clearly Borel.
\end{proof}
\begin{remark}
Observe that the only geometric property of Banach spaces that we used in the proof, besides that Banach spaces are cones so that $\rho_{HL}$ and $\rho_N$ agree on them (see Remark \ref{remark:ConePropertyOfBanach}), was that Banach spaces are geodesic metric spaces; that is, metric spaces $M$ such that for each pair of points $(x,y)\in M^2$, there is an isometric embedding $\iota: [0,d(x,y)]\rightarrow M$ with $\iota(0)=x$ and $\iota(d(x,y))=y$. In fact, a weaker assumption is sufficient, that they are length spaces, i.e. between every two points $x,y$ there is a path of length $d(x,y)+\varepsilon$, where $\varepsilon>0$ is arbitrary. Therefore it follows from the proof of Theorem \ref{thm:ReductionHLdistToGH} that there is a reduction from $\rho_{HL}$ (or $\rho_N$) on cones that are length spaces to $\rho_{GH}$ on metric spaces.
\end{remark}
Finally, we present the proof of the reduction that involves the Kadets distance.
\begin{thm}\label{thm:KadetsToGH}
There is an injective Borel-uniformly continuous reduction from $\rho_K$ on $\Banach$ to $\rho_{GH}$ on $\Met$.

Moreover, the reduction is not only Borel-uniformly continuous, but also Borel-Lipschitz on small distances.
\end{thm}
We first need the following lemma.
\begin{lemma}\label{lem:KadetsGivesBijections}
Let $X$ and $Y$ be two separable Banach spaces and fix countable dense subsets $(x_i)_i$ and $(y_i)_i$ of the spheres $S_X$ and $S_Y$ respectively. Then $\rho_K(X,Y)<\varepsilon$, for some $\varepsilon>0$, implies that there exists a bijection $\pi\in S_\infty$ such that for every finite $F\subseteq \Nat$ and every $(\delta_i)_{i\in F} \in\{-1,1\}^F$ we have
\[
\left| \Big\|\sum_{i\in F} \delta_i x_i\Big\|_X - \Big\|\sum_{i\in F} \delta_i y_{\pi(i)}\Big\|_Y\right|<2|F|\varepsilon.
\]
\end{lemma}
\begin{proof}
We may suppose that $X$ and $Y$ are subspaces of a Banach space $Z$ and that we have $\rho_H^Z(B_X,B_Y)<\varepsilon$. First we claim that for every $x\in S_X$ there exists $y\in S_Y$ such that $\|x-y\|<2\varepsilon$. Analogously, for every $y\in S_Y$ there exists such $x\in S_X$. Indeed, by definition for every $x\in S_X$ there exists $y'\in B_Y$ with $\|x-y'\|<\varepsilon$. So we can take $y=y'/\|y'\|$ and we have $\|y-y'\|<\varepsilon$, so we are done by the triangle inequality. Now since $S_X$ and $S_Y$ are perfect metric spaces, by a back-and-forth argument (see e.g. the proof of Lemma \ref{lem:GHPerfectSpaces}), we get a bijection $\pi\in S_\infty$ such that for every $i\in\Nat$ we have $\|x_i-y_{\pi(i)}\|<2\varepsilon$. We claim that $\pi$ is as desired.

Take any finite subset $F\subseteq \Nat$ and $(\delta_i)_{i\in F}\in\{-1,1\}^F$. Then we have
\[\begin{split}
\left| \Big\|\sum_{i\in F} \delta_i x_i\Big\| - \Big\|\sum_{i\in F} \delta_i y_{\pi(i)}\Big\|\right| & \leq \Big\|\sum_{i\in F} \delta_i(x_i-y_{\pi(i)})\Big\|
 \leq \sum_{i\in F} \|x_i-y_{\pi(i)}\| <\\
 & <2|F|\varepsilon,
\end{split}\]
and we are done.
\end{proof}
\begin{proof}[Proof of Theorem \ref{thm:KadetsToGH}.]
The structure of the proof is the following. First, we describe a construction which to each separable Banach space $X$ assigns a metric space $M_X$. Next, we show that for every two separable Banach spaces $X$ and $Y$ we have $\rho_{GH}(M_X,M_Y) \leq 2\rho_{K}(X,Y)$ and
\[
	\rho_{GH}(X,Y) < 1 \implies \rho_{K}(X,Y) \leq 17\rho_{GH}(M_X,M_Y).
\]
Finally, we show it is possible to make such an assignment in a Borel way.

Let $X$ be a separable Banach space. Fix a countable dense subset $D_X=\{x_i\setsep i\in\Nat\}\subseteq S_X$ of the unit sphere of $X$ that is symmetric, that is, for every $x\in D_X$ we also have $-x\in D_X$. Set $M_X=D_X\cup\{p_{F,k}\setsep F\in [\Nat]^{<\omega}\setminus\{\emptyset\}, k\in F\}$. We define a metric $d_X$ on $M_X$ as follows: $$d_X(x_i,x_j)=\|x_i-x_j\|_X,$$ $$d_X(x_i,p_{F,k})=10+\|x_i-x_k\|_X,$$ $$d_X(p_{F,i},p_{F,j})=15+\frac{\|\sum_{k\in F}x_k\|_X}{|F|},\quad i\neq j\in F,$$ $$d_X(p_{F,i},p_{G,j})=20,\quad F\neq G.$$

Fix separable Banach spaces $X$ and $Y$. The space $M_Y=D_Y\cup\{q_{F,k}\setsep F\in [\Nat]^{<\omega}\setminus\{\emptyset\}, k\in F\}$, where $D_Y=\{y_i\setsep i\in\Nat\}\subseteq S_Y$ is symmetric countable dense, is constructed analogously as $M_X$.

We claim that for every $\varepsilon > 0$ with $\rho_K(X,Y) < \varepsilon$ we have $\rho_{GH}(M_X,M_Y) \leq 2\varepsilon$. Indeed, fix $\varepsilon > 0$ with $\rho_K(X,Y) < \varepsilon$ and use Lemma \ref{lem:KadetsGivesBijections} applied to countable dense sequences $(x_i)_i$ and $(y_i)_i$ of the spheres $S_X$ and $S_Y$ respectively. The bijection $\pi$ from Lemma~\ref{lem:KadetsGivesBijections} induces a bijection $\phi: M_X\rightarrow M_Y$ defined as follows: $$\phi(x_i)=y_{\pi(i)},$$ $$\phi(p_{F,j})=q_{\pi[F],\pi(j)}.$$
We claim that for every $x,y\in M_X$ we have $|d_X(x,y)-d_Y(\phi(x),\phi(y))|<4\varepsilon$, i.e. $M_X\simeq_{4\varepsilon} M_Y$. We consider several cases:

\medskip
\noindent{\bf Case 1.} $(x,y)=(x_i,x_j)$ for some $i,j\in\Nat$: then we have $$|d_X(x_i,x_j)-d_Y(y_{\pi(i)},y_{\pi(j)})|=| \|x_i-x_j\|_X-\|y_{\pi(i)}-y_{\pi(j)}\|_Y |<4\varepsilon.$$
\noindent{\bf Case 2.} $x=x_i$, $y=p_{F,k}$ for some $i\in\Nat$, $F\subseteq \Nat$, $k\in F$: then we have $$|d_X(x_i,p_{F,k})-d_Y(y_{\pi(i)},q_{\pi[F],\pi(k)})|=|\|x_i-x_k\|_X-\|y_{\pi(i)}-y_{\pi(k)}\|_Y|<4\varepsilon.$$
\noindent{\bf Case 3.} $x=p_{F,j}$, $y=p_{F,k}$ for some $F\subseteq \Nat$, $j\neq k\in F$: then we have 
\[\begin{split}
|d_X(p_{F,j},p_{F,k})-d_Y(q_{\pi[F],\pi(j)},q_{\pi[F],\pi(k)})| & = \frac{| \|\sum_{i\in F}x_i\|_X-\|\sum_{i\in F} y_{\pi(i)}\|_Y|}{|F|}\\
& < \frac{2|F|\varepsilon}{|F|}=2\varepsilon.
\end{split}\]
\noindent{\bf Case 4.} $x=p_{F,i}$, $y=p_{G,j}$, for $F\neq G\subseteq \Nat$, $i\in F$, $j\in G$: then we have $$|d_X(p_{F,i},p_{G,j})-d_Y(q_{\pi[F],\pi(i)},q_{\pi[G],\pi(j)})|=20-20=0.$$
Hence, $d_X\simeq_{4\varepsilon} d_Y$ and, by Lemma \ref{lem:lehciImplikace}, we get $\rho_{GH}(M_X,M_Y)\leq 2\varepsilon$ which proves the claim.

\medskip
Conversely, suppose now that $\rho_{GH}(M_X,M_Y)<\varepsilon$, where $\varepsilon\in(0,1)$. By Fact \ref{fact:GHbyCorrespondences} there exists a correspondence $\Corr\subseteq M_X\times M_Y$ such that for every $x,y\in M_X$ and $x',y'\in M_Y$, if $x\Corr x'$ and $y\Corr y'$, then $|d_X(x,y)-d_Y(x',y')|<2\varepsilon$. Pick some $i\neq j\in\Nat$, a finite subset $F\subseteq \Nat$, $k\neq k'\in F$. Set $u_1=x_i$, $u_2=x_j$, $u_3=p_{F,k}$, $u_4=p_{F,k'}$. We find elements $v_1,\ldots,v_4\in M_Y$ such that $u_i\Corr v_i$ for $i\leq 4$. We get the following observations:
\begin{itemize}
\item Since $d_X(u_1,u_2)\in [0,2]$ we get that $d_Y(v_1,v_2)\in [0,4]$, so we deduce that for every $n\in\Nat$ and every $y\in M_Y$ such that $x_n\Corr y$ we have $y=y_m$ for some $m\in\Nat$. Conversely, for every $n\in\Nat$ and every $x\in M_X$ such that $x\Corr y_n$ we have $x=x_m$ for some $m\in\Nat$.
\item Since $d_X(u_3,u_4)\in [15,16]$ we get that $d_Y(v_3,v_4)\in [13,18]$. So we deduce that for every finite subsets $G,G'\subseteq \Nat$ and $l\in G$, $l'\in G'$, and every $y,y'\in M_Y$ such that $p_{G,l}\Corr y$ and $p_{G',l'}\Corr y'$ there are finite subsets $H,H'\subseteq \Nat$ and $h\in H$, $h'\in H'$ such that $y=p_{H,h}$, $y'=p_{H',h'}$ and $G=G'$ if and only if $H=H'$.
\end{itemize}
To summarize, $\Corr$ induces a bijection $\phi$ between $M_X\setminus (x_i)_i$ and $M_Y\setminus (y_i)_i$. Moreover, for every finite $F\subseteq \Nat$ there is a unique finite set, which we shall denote by $\varphi(F)$, such that $\phi$ is a bijection between $\{p_{F,i}\setsep i\in F\}$ and $\{q_{\varphi(F),j}\setsep j\in \varphi(F)\}$. For every $i\in F$, by $\varphi_F(i)$ we shall the denote the element $i'\in \varphi(F)$ such that $q_{\varphi(F),i'}=\phi(p_{F,i})$.

On the other hand, $\Corr$, when restricted on $D_X\times D_Y$, is a correspondence between $D_X$ and $D_Y$ witnessing that $\rho_{GH}(D_X,D_Y)\leq\varepsilon$. Since $d_X\upharpoonright D_X$ and $d_Y\upharpoonright D_Y$ are perfect metric spaces, by a back-and-forth argument (see e.g. the proof of Lemma \ref{lem:GHPerfectSpaces}) we construct a bijection $\phi'\subseteq \Corr$ between $D_X$ and $D_Y$ such that $|d_X(x_i,x_j)-d_Y(\phi'(x_i),\phi'(x_j))|<2\varepsilon$. Taking the union of the bijections $\phi$ and $\phi'$ we get a bijection, which we shall still denote by $\phi$, between $M_X$ and $M_Y$ such that for every $x,y\in M_X$, $|d_X(x,y)-d_Y(\phi(x),\phi(y))|<2\varepsilon$.

Pick now an arbitrary finite $F\subseteq \Nat$. We want to estimate the expression
\[
\left| \Big\|\sum_{k\in F} x_k\Big\|_X - \Big\|\sum_{k\in F} \phi(x_k)\Big\|_Y\right|.
\]
We may suppose $F$ contains at least two elements. Take any $i\neq i'\in F$ and set $G=\varphi(F)$ and $j=\varphi_F(i)$, $j'=\varphi_F(i')$. Since we have $d_X(p_{F,i},p_{F,i'})=15+\|\sum_{k\in F} x_k\|_X/|F|$ and $d_Y(q_{G,j},q_{G,j'})=15+\|\sum_{k\in G} y_k\|_Y/|G|$, and moreover $| d_X(p_{F,i},p_{F,i'})-d_Y(q_{G,j},q_{G,j'})|<2\varepsilon$ we get
\[
\left| \Big\|\sum_{k\in F} x_k\Big\|_X - \Big\|\sum_{k\in G} y_k\Big\|_Y\right|<2\varepsilon|F|.
\]
So we try to estimate
\[
\left|\Big\|\sum_{k\in G} y_k\Big\|_Y - \Big\|\sum_{k\in F} \phi(x_k)\Big\|_Y\right|.
\]
Pick any $k\in F$ and let $k'=\varphi_F(k)\in G$. We have $d_X(x_k,p_{F,k})=10$ and $d_Y(\phi(x_k),q_{G,k'})=10+\|\phi(x_k)-y_{k'}\|_Y$. Therefore, since
\[
\left| d_X(x_k,p_{F,k})-d_Y(\phi(x_k),q_{G,k'})\right|= \left|d_X(x_k,p_{F,k})-d_Y(\phi(x_k),\phi(p_{F,k})\right|<2\varepsilon,
\]
we get that $\|\phi(x_k)-y_{k'}\|_Y<2\varepsilon$. This implies that 
\[
\left|\Big\|\sum_{k\in G} y_k\Big\|_Y - \Big\|\sum_{k\in F} \phi(x_k)\Big\|_Y\right|<2\varepsilon|F|,
\]
which in turn implies that
\[
\left| \Big\|\sum_{k\in F} x_k\Big\|_X - \Big\|\sum_{k\in F} \phi(x_k)\Big\|_Y\right|<4\varepsilon|F|.
\]
Note that the last inequality in particular implies that $\phi$ is almost symmetric in the following sense: Pick any $x\in D_X$. Since also $-x\in D_X$, the previous inequality implies
\begin{equation}\label{ineq:AlmostSymmetry}
\Big|\|x - x\|_X-\|\phi(x)+\phi(-x)\|_Y\Big|=\|-\phi(x)-\phi(-x)\|_Y<8\varepsilon.
\end{equation}
We set $E=\{e\in X\setsep e=qx, q\in\Rat^+\cup\{0\},x\in D_X\}$ and $F=\{f\in Y\setsep f=qy, q\in\Rat^+\cup\{0\},y\in D_Y\}$. Clearly $ E $ and $F $ are $\Rat$-homogeneous dense subsets of $X$ and $Y$ respectively. We define a correspondence $\Corr_0\subseteq E\times F$, in fact a bijection, such that $e\Corr_0 f$ if and only if there are $x\in D_X$ and $q\in\Rat^+$ such that $e=qx$, $f=q\phi(x)$. So for every pair $(e,f)$ such that $e\Corr_0 f$ we have $\|e\|_X=\|f\|_Y$. We now claim that $\Corr_0$ is such that 
\begin{equation}\label{ineq:KadetsCorr}
\Bigg| \Big\|\sum_{i\leq n} u_i\Big\|_X-\Big\|\sum_{i\leq n} u'_i\Big\|_Y\Bigg|\leq 8\varepsilon\Big(\sum_{i\leq n} \|u_i\|_X\Big)
\end{equation} for all $(u_i)_i\subseteq E$ and $(u'_i)_i\subseteq F$, where for all $i\leq n$ we have $u_i\Corr_0 u'_i$.

Fix such  a sequence $(u_i)_{i\leq n}\subseteq E$. The corresponding sequence $(u'_i)_i$ is then determined uniquely. First we claim that without loss of generality we may suppose that $(u_i)_{i\leq n}\subseteq D_X$. Indeed, by the homogeneity of the inequality above, we may assume that each $u_i$ is a positive integer multiple of some $x\in D_X$. Since we allow repetitions in the sequence $(u_i)_{i\leq n}$, each element of the form $kx$, where $k\in\Nat$ and $x\in D_X$ can be replaced by $k$-many repetitions of the element $x$.

Next we show how we may approximate the sequence $(u_i)_i$, in which we allow repetitions, by a sequence $(a_i)_i\subseteq D_X$ in which we do not allow repetitions. For each $i\leq n$, choose some $a_i\in D_X$ such that $\|a_i-u_i\|_X<\varepsilon$. Let $(a'_i)_i \subseteq D_Y$ be the elements such that for all $i\leq n$ we have $a_i \Corr_0 a'_i$. Since by the assumption we have $\left| \|u_i-a_i\|_X-\|u'_i-a'_i\|_Y\right|<2\varepsilon$, we get $\|u'_i-a'_i\|_Y<3\varepsilon$. Notice that for such sequences we get, by the computations above, \[
\Bigg| \Big\|\sum_{i\leq n} a_i\Big\|_X-\Big\|\sum_{i\leq n} a'_i\Big\|_Y\Bigg|\leq 4\varepsilon n.
\]
This, together with the inequalities $\|u_i-a_i\|_X<\varepsilon$ and $\|u'_i-a'_i\|_Y<3\varepsilon$ implies that 
\[
\Bigg| \Big\|\sum_{i\leq n} u_i\Big\|_X-\Big\|\sum_{i\leq n} u'_i\Big\|_Y\Bigg|\leq 8\varepsilon n,
\]
which proves the inequality \eqref{ineq:KadetsCorr}.

Before we are in the position to apply Lemma \ref{lem:KadetsDescription} we need to guarantee that $\Corr_0$ is $\Rat$-homogeneous. Note that so far it is only closed under multiplication by positive rationals. This will be fixed in the last step.

Set $\bar \Corr=\Corr_0\cup -\Corr_0$, where $-\Corr_0=\{(x,y)\setsep (-x)\Corr_0 (-y)\}$. Now $\bar \Corr$ is clearly $\Rat$-homogeneous. Pick now an arbitrary sequence $(u_i)_{i\leq n}\subseteq E$ and a sequence $(u'_i)_i\subseteq F$  such that for all $i\leq n$ we have $u_i\bar \Corr u'_i$. For each $i\leq n$, pick $u''_i\in F$ such that $u_i\Corr_0 u''_i$. Either $u''_i=u'_i$, or by \eqref{ineq:AlmostSymmetry} we get $\|u''_i-u'_i\|_Y\leq 8\varepsilon\|u_i\|_X$. From these inequalities and from \eqref{ineq:KadetsCorr}, which gives us 
\[
\Bigg| \Big\|\sum_{i\leq n} u_i\Big\|_X-\Big\|\sum_{i\leq n} u''_i\Big\|_Y\Bigg|\leq 8\varepsilon\Big(\sum_{i\leq n} \|u_i\|_X\Big),
\]
we get the estimate
\[
\Bigg| \Big\|\sum_{i\leq n} u_i\Big\|_X-\Big\|\sum_{i\leq n} u'_i\Big\|_Y\Bigg|\leq 16\varepsilon\Big(\sum_{i\leq n} \|u_i\|_X\Big).
\]
The application of Lemma \ref{lem:KadetsDescription} then gives us that $\rho_K(X,Y)<17\varepsilon$.

\medskip
It remains to see that it is possible to find an injective and Borel map $f: \Banach\rightarrow\Met$ such that $f(X)$ is isometric to $(M_X,d_X)$ for every $X\in \Banach$. Each Banach space $X$ is coded as a norm $\Norm_X\in\Banach$ which is defined on a countable infinite-dimensional vector space over $\Rat$ denoted by $V$. First we need to select in a Borel way a countable dense symmetric subset of $S_X$. Pick $D\subseteq V\setminus\{0\}$ such that $D$ contains exactly one element of $\{tv\setsep t>0\}$ for every $v\in V\setminus\{0\}$. Fix some bijection $g:D\rightarrow \Nat$ and define a metric $d'_X$ on $\Nat$ as follows: for $n,m\in \Nat$ we set
\[
d'_X(n,m)=\left\|\frac{g^{-1}(n)}{\|g^{-1}(n)\|_X}-\frac{g^{-1}(m)}{\|g^{-1}(m)\|_X} \right\|_X.
\]
This corresponds to selecting a countable dense symmetric subset of $S_X$ with the metric inherited from $\Norm_X$. Clearly, the assignment $\Norm_X\to d'_X$ is injective and Borel. Then we only add to $\Nat$ a fixed countable set $\{p_{F,k}\setsep F\in [\Nat]^{<\omega}\setminus\{\emptyset\}, k\in F\}$ and define the metric $d_X$ on the union of these two countable sets using the norm $\Norm_X$. Finally, we reenumerate this countable set so that $d_X$ is defined on $\Nat$, and so belongs to $\Met$. That is clearly one-to-one and Borel.
\end{proof}

\section{Reductions from pseudometrics on \texorpdfstring{$\Met$}{M} to pseudometrics on \texorpdfstring{$\Banach$}{B}}\label{sectionReduction3}
This section is devoted to the proof of the following result.

\begin{thm} \label{thmGHtoK}
There is an injective Borel-uniformly continuous reduction from $ \rho_{GH} $ on $ \Met_{1/2}^{1} $ to each of the distances $ \rho_{K}, \rho_{BM}, \rho_{L}, \rho_{U}, \rho_{N}, \rho_{GH}^{\Banach} $ on $ \Banach $.

Moreover, for the distances $ \rho_{K} $ and $ \rho_{BM} $, the reduction is not only Borel-uniformly continuous, but also Borel-Lipschitz.
\end{thm}

The definition of our reduction is based on a simple geometric idea of renorming of the Hilbert space $ \ell_{2} $. However, the proof that the idea works is technical and splits into many steps.

Let us denote by $ e_{n} $ the sequence in $ \ell_{2} $ that has $ 1 $ at the $ n $-th place and $ 0 $ elsewhere. Let us moreover denote
$$ e_{n, m} = \frac{1}{\sqrt{2}}(e_{n}+e_{m}), \quad \{ n, m \} \in [\mathbb{N}]^{2}. $$
The following fact can be verified by a simple computation.

\begin{lemma} \label{lemmsep}
For $ \{ n, m \}, \{ n', m' \} \in [\mathbb{N}]^{2} $, we have
\begin{align*}
\Vert e_{n,m} - e_{n',m'} \Vert_{\ell_{2}} & = 1 & & \textrm{if } |\{ n, m \} \cap \{ n', m' \}| = 1, \\
\Vert e_{n,m} + e_{n',m'} \Vert_{\ell_{2}} & = \sqrt{3} & & \textrm{if } |\{ n, m \} \cap \{ n', m' \}| = 1, \\
\Vert e_{n,m} \pm e_{n',m'} \Vert_{\ell_{2}} & = \sqrt{2} & & \textrm{if } \{ n, m \} , \{ n', m' \} \textrm{ are disjoint}.
\end{align*}
Hence, the set of all vectors $ \pm e_{n,m} $ is $ 1 $-separated.
\end{lemma}

Let us fix numbers $ \alpha $ and $ \delta $ such that
$$ 1 < \alpha < \alpha + \delta \leq \frac{200}{199}. $$
For every $ f : [\mathbb{N}]^{2} \to [0, 1] $, we define an equivalent norm $ \Vert \cdot \Vert_{f} $ on $ \ell_{2} $ by
$$ \Vert x \Vert_{f} = \sup \Big( \{ \Vert x \Vert_{\ell_{2}} \} \cup \Big\{ \frac{1}{\sqrt{2}} \cdot \big( \alpha + \delta \cdot f(n, m) \big) \cdot |x_{n}+x_{m}| : n \neq m \Big\} \Big) $$
for $ x = (x_{n})_{n=1}^{\infty} \in \ell_{2} $. This is an equivalent norm indeed, as $ \frac{1}{\sqrt{2}} \cdot |x_{n}+x_{m}| = | \langle x, e_{n,m} \rangle | \leq \Vert x \Vert_{\ell_{2}} $, and consequently
$$ \Vert x \Vert_{\ell_{2}} \leq \Vert x \Vert_{f} \leq \frac{200}{199} \Vert x \Vert_{\ell_{2}}, \quad x \in \ell_{2}. $$
Let us define
$$ P_{n,m} = \Big\{ x \in \ell_{2} : \Vert x \Vert_{\ell_{2}} \leq \frac{1}{\sqrt{2}} \cdot \big( \alpha + \delta \cdot f(n, m) \big) \cdot (x_{n}+x_{m}) \Big\}. $$
It follows from the following lemma that any non-zero $ x \in \ell_{2} $ belongs to at most one set $ \pm P_{n,m} $.

\begin{lemma} \label{lemmPnm}
Let us denote $ h = \alpha + \delta \cdot f(n, m) $. If $ \Vert x \Vert_{\ell_{2}} = 1 $, then
$$ x \in P_{n,m} \quad \Leftrightarrow \quad \Vert x - e_{n,m} \Vert_{\ell_{2}} \leq \sqrt{\frac{2(h-1)}{h}}. $$
In particular,
$$ x \in P_{n,m} \quad \Rightarrow \quad \Vert x - e_{n,m} \Vert_{\ell_{2}} \leq \frac{1}{10}. $$
\end{lemma}

\begin{proof}
We compute
\begin{align*}
\Vert x - e_{n,m} \Vert_{\ell_{2}} \leq \sqrt{\frac{2(h-1)}{h}} \; & \Leftrightarrow \; \Vert x \Vert_{\ell_{2}}^{2} - 2 \langle x, e_{n,m} \rangle + \Vert e_{n,m} \Vert_{\ell_{2}}^{2} \leq \frac{2(h-1)}{h} \\
 & \Leftrightarrow \; 2 \langle x, e_{n,m} \rangle \geq \frac{2}{h} \\
 & \Leftrightarrow \; \frac{1}{\sqrt{2}} \cdot h \cdot (x_{n}+x_{m}) \geq 1 \\
 & \Leftrightarrow \; x \in P_{n,m}.
\end{align*}
Finally, since $ h \leq \alpha + \delta \leq \frac{200}{199} $, we have $ \sqrt{\frac{2(h-1)}{h}} \leq \frac{1}{10} $.
\end{proof}

Our proof of Theorem~\ref{thmGHtoK} is based on the following technical lemma.

\begin{lemma} \label{lemmGHtoK}
Let $ f, g : [\mathbb{N}]^{2} \to [0, 1] $. If $ \rho_{K}((\ell_{2}, \Vert \cdot \Vert_{f}), (\ell_{2}, \Vert \cdot \Vert_{g})) < \eta $ for some $ \eta $ satisfying $ 0 < \eta < \frac{1}{100}, \eta < \frac{1}{10} \cdot \frac{\sqrt{\alpha^{2}-1}}{\alpha} $ and $ \eta \leq \frac{1}{2}(1 - \frac{1}{\alpha}) $, then
$$ \exists \pi \in S_{\infty} \forall \{ n, m \} \in [\mathbb{N}]^{2} : \big| g(\pi(n), \pi(m)) - f(n, m) \big| < \frac{3}{\delta} \cdot \eta. $$
\end{lemma}

Due to the assumption of the lemma, we can pick a Banach space $ Z $ and linear isometries $ I : (\ell_{2}, \Vert \cdot \Vert_{f}) \to Z $ and $ J : (\ell_{2}, \Vert \cdot \Vert_{g}) \to Z $ such that
$$ \rho_{H}^{Z} \big( I(B_{(\ell_{2}, \Vert \cdot \Vert_{f})}), J(B_{(\ell_{2}, \Vert \cdot \Vert_{g})}) \big) < \eta. $$

We need to prove the following claim first.

\begin{claim} \label{claimGHtoK}
For every $ \{ n, m \} \in [\mathbb{N}]^{2} $, there are $ \{ k, l \} \in [\mathbb{N}]^{2} $ and $ s \in \{ -1, 1 \} $ such that
$$ \big\Vert Ie_{n,m} - s \cdot Je_{k,l} \big\Vert_{Z} < \frac{1}{7} $$
and, moreover,
$$ f(n, m) - g(k, l) < \frac{3}{\delta} \cdot \eta. $$
\end{claim}

\begin{proof}
We prove the claim in eight steps. Let us fix $ \{ n, m \} \in [\mathbb{N}]^{2} $ and keep the notation $ h = \alpha + \delta \cdot f(n, m) $ throughout the proof. Analogously as above, we define
$$ Q_{k,l} = \Big\{ x \in \ell_{2} : \Vert x \Vert_{\ell_{2}} \leq \frac{1}{\sqrt{2}} \cdot \big( \alpha + \delta \cdot g(k, l) \big) \cdot (x_{k}+x_{l}) \Big\}. $$

Step 1: We show that there is $ x $ orthogonal to $ e_{n,m} $ such that $ \Vert x \Vert_{\ell_{2}} = 1 $ and $ \Vert Ix \mp Je_{k,l} \Vert_{Z} \geq \frac{1}{2} $ for all $ k \neq l $. This is an easy consequence of the fact that the vectors $ \pm Je_{k, l} $ are $ 1 $-separated (which follows from Lemma~\ref{lemmsep} and from $ \Vert \cdot \Vert_{g} \geq \Vert \cdot \Vert_{\ell_{2}} $). Indeed, let $ E \subseteq \ell_{2} $ be a two-dimensional subspace orthogonal to $ e_{n,m} $. Let us pick $ x \in S_{E} $. If $ \Vert Ix \mp Je_{k,l} \Vert_{Z} \geq \frac{1}{2} $ for all $ k\neq l $, we are done. In the opposite case, there is a point $ w = Je_{k,l} $ or $ w = -Je_{k,l} $ for which $ \Vert Ix - w \Vert_{Z} < \frac{1}{2} $. Since $ I(S_{E}) $ is a closed curve in $ Z $ with diameter at least $ 2 $, we can find $ x' \in S_{E} $ such that $ \Vert Ix' - w \Vert_{Z} = \frac{1}{2} $. Then $ x' $ works, as the distance of $ Ix' $ to other vectors is at least $ 1 - \frac{1}{2} $ by the triangle inequality.

Step 2: We denote
$$ p_{+} = \frac{1}{h}e_{n,m} + \frac{\sqrt{h^{2}-1}}{h}x, \quad p_{-} = \frac{1}{h}e_{n,m} - \frac{\sqrt{h^{2}-1}}{h}x. $$
It is easy to see that
$$ \Vert p_{+} \Vert_{f} = \Vert p_{-} \Vert_{f} = 1. $$
Let us choose $ q_{+} $ and $ q_{-} $ with $ \Vert q_{+} \Vert_{g} \leq 1, \Vert q_{-} \Vert_{g} \leq 1, $ satisfying
$$ \Vert Ip_{+} - Jq_{+} \Vert_{Z} < \eta, \quad \Vert Ip_{-} - Jq_{-} \Vert_{Z} < \eta. $$

Step 3: We show that
$$ \frac{\Vert q_{+}+q_{-} \Vert_{g}}{2} > 1-\eta. $$
Since $ \Vert (Ip_{+} + Ip_{-})/2 - (Jq_{+} + Jq_{-})/2 \Vert_{Z} < \eta $, we have
$$ \frac{\Vert q_{+}+q_{-} \Vert_{g}}{2} = \Big\Vert \frac{1}{2} (Jq_{+} + Jq_{-}) \Big\Vert_{Z} > \Big\Vert \frac{1}{2} (Ip_{+} + Ip_{-}) \Big\Vert_{Z} - \eta = \Big\Vert \frac{1}{h}e_{n,m} \Big\Vert_{f} - \eta. $$
As $ \Vert e_{n,m} \Vert_{f} = h $, the desired inequality follows.

Step 4: We show that
$$ \Vert q_{+} - q_{-} \Vert_{g} = \Vert q_{+} - q_{-} \Vert_{\ell_{2}}, $$
by proving that $ q_{+} - q_{-} $ does not belong to any $ \pm Q_{k,l} $. Let us denote
$$ u = \frac{h}{2\sqrt{h^{2}-1}} \cdot (q_{+} - q_{-}) \quad \textrm{and} \quad z = \frac{1}{\Vert u \Vert_{\ell_{2}}} \cdot u. $$
Then
\begin{align*}
\Vert Ju - Ix \Vert_{Z} & = \frac{h}{2\sqrt{h^{2}-1}} \cdot \big\Vert (Jq_{+} - Jq_{-}) - (Ip_{+} - Ip_{-}) \big\Vert_{Z} \\
 & < \frac{h}{2\sqrt{h^{2}-1}} \cdot 2\eta = \frac{h}{\sqrt{h^{2}-1}} \cdot \eta \leq \frac{\alpha}{\sqrt{\alpha^{2}-1}} \cdot \eta < \frac{1}{10}
\end{align*}
by an assumption of Lemma~\ref{lemmGHtoK}. By the choice of $ x $ (Step 1), we obtain for all $ k \neq l $ that
$$ \Vert u \mp e_{k,l} \Vert_{g} = \Vert Ju \mp Je_{k,l} \Vert_{Z} > \Vert Ix \mp Je_{k,l} \Vert_{Z} - \frac{1}{10} \geq \frac{1}{2} - \frac{1}{10} = \frac{2}{5}. $$
Also, it is easy to check that
$$ | \Vert u \Vert_{\ell_{2}} - 1 | < \frac{1}{199} + \frac{1}{10}. $$
Indeed, we have $ \Vert u \Vert_{\ell_{2}} \leq \Vert u \Vert_{g} < \Vert x \Vert_{f} + \frac{1}{10} \leq \frac{200}{199} \cdot \Vert x \Vert_{\ell_{2}} + \frac{1}{10} = 1 + \frac{1}{199} + \frac{1}{10} $ and $ \Vert u \Vert_{\ell_{2}} \geq \frac{199}{200} \cdot \Vert u \Vert_{g} > \frac{199}{200} \cdot (\Vert x \Vert_{f} - \frac{1}{10}) \geq \frac{199}{200} \cdot (\Vert x \Vert_{\ell_{2}} - \frac{1}{10}) = \frac{199}{200} \cdot (1 - \frac{1}{10}) > 1 - (\frac{1}{199} + \frac{1}{10}) $.

Since $ z = u - (\Vert u \Vert_{\ell_{2}} - 1)z $ and $ \Vert z \Vert_{\ell_{2}} = 1 $, we obtain for all $ k \neq l $ that
$$ \Vert z \mp e_{k,l} \Vert_{\ell_{2}} \geq \Vert u \mp e_{k,l} \Vert_{\ell_{2}} - |\Vert u \Vert_{\ell_{2}} - 1| \Vert z \Vert_{\ell_{2}} > \frac{199}{200} \cdot \frac{2}{5} - \Big( \frac{1}{199} + \frac{1}{10} \Big) > \frac{1}{10}. $$
By Lemma~\ref{lemmPnm}, $ z $ does not belong to $ \pm Q_{k,l} $. The same holds for $ q_{+} - q_{-} $, as it is a multiple of $ z $.

Step 5: We show that
$$ \frac{\Vert q_{+}+q_{-} \Vert_{\ell_{2}}}{2} < \frac{1}{h} + \eta. $$
By the parallelogram law,
$$ \Vert q_{+} + q_{-} \Vert_{\ell_{2}}^{2} + \Vert q_{+} - q_{-} \Vert_{\ell_{2}}^{2} = 2 \Vert q_{+} \Vert_{\ell_{2}}^{2} + 2 \Vert q_{-} \Vert_{\ell_{2}}^{2} \leq 2 \Vert q_{+} \Vert_{g}^{2} + 2 \Vert q_{-} \Vert_{g}^{2} \leq 4. $$
Using the conclusion of the previous step,
$$ \Vert q_{+} - q_{-} \Vert_{\ell_{2}} = \Vert q_{+} - q_{-} \Vert_{g} > \Vert p_{+} - p_{-} \Vert_{f} - 2\eta \geq \Vert p_{+} - p_{-} \Vert_{\ell_{2}} - 2\eta = 2 \cdot \frac{\sqrt{h^{2}-1}}{h} - 2\eta, $$
$$ \Vert q_{+} - q_{-} \Vert_{\ell_{2}}^{2} > 4 \cdot \frac{h^{2}-1}{h^{2}} - 8 \cdot \frac{\sqrt{h^{2}-1}}{h} \cdot \eta + 4\eta^{2} > 4 \cdot \frac{h^{2}-1}{h^{2}} - 8 \cdot \frac{1}{h} \cdot \eta $$
and
$$ \Vert q_{+} + q_{-} \Vert_{\ell_{2}}^{2} < 4 - 4 \cdot \frac{h^{2}-1}{h^{2}} + 8 \cdot \frac{1}{h} \cdot \eta < \frac{4}{h^{2}} + 8 \cdot \frac{1}{h} \cdot \eta + 4\eta^{2} = 4 \cdot \Big( \frac{1}{h} + \eta \Big)^{2}. $$
The desired inequality follows.

Step 6: We realize that $ q_{+}+q_{-} $ belongs to some $ Q_{k,l} $ or $ -Q_{k,l} $. In the opposite case, we obtain $ \Vert q_{+} + q_{-} \Vert_{g} = \Vert q_{+} + q_{-} \Vert_{\ell_{2}} $ and
$$ 1-\eta < \frac{\Vert q_{+} + q_{-} \Vert_{g}}{2} = \frac{\Vert q_{+}+q_{-} \Vert_{\ell_{2}}}{2} < \frac{1}{h} + \eta \leq \frac{1}{\alpha} + \eta, $$
which is disabled by an assumption of Lemma~\ref{lemmGHtoK}.

Step 7: Let $ q_{+}+q_{-} $ belong to $ s \cdot Q_{k,l} $, where $ s \in \{ -1, 1 \} $. We show that
$$ \big\Vert Ie_{n,m} - s \cdot Je_{k,l} \big\Vert_{Z} < \frac{1}{7} $$
for such $ k, l $ and $ s $. To show this, we denote
$$ a = \frac{1}{\Vert q_{+}+q_{-} \Vert_{\ell_{2}}} \cdot (q_{+}+q_{-}). $$
Then $ a $ belongs to $ s \cdot Q_{k,l} $ as well. Lemma~\ref{lemmPnm} provides
$$ \Vert a - s \cdot e_{k,l} \Vert_{\ell_{2}} \leq \frac{1}{10}, $$
and so
$$ \Vert Ja - s \cdot Je_{k,l} \Vert_{Z} = \Vert a - s \cdot e_{k,l} \Vert_{g} \leq \frac{200}{199} \cdot \frac{1}{10}. $$
Since $ \Vert (Ip_{+} + Ip_{-})/2 - (Jq_{+} + Jq_{-})/2 \Vert_{Z} < \eta < \frac{1}{100} $, we obtain
$$ \Big\Vert \frac{1}{h} Ie_{n,m} - \frac{\Vert q_{+}+q_{-} \Vert_{\ell_{2}}}{2} Ja \Big\Vert_{Z} < \frac{1}{100} $$
and
$$ \Vert Ie_{n,m} - Ja \Vert_{Z} < \frac{1}{100} + \Big( 1 - \frac{1}{h} \Big) \Vert Ie_{n,m} \Vert_{Z} + \Big( 1 - \frac{\Vert q_{+}+q_{-} \Vert_{\ell_{2}}}{2} \Big) \Vert Ja \Vert_{Z}. $$
Since $ \Vert e_{n,m} \Vert_{f} = h, \Vert a \Vert_{\ell_{2}} = 1 $ and
$$ \frac{\Vert q_{+}+q_{-} \Vert_{\ell_{2}}}{2} \geq \frac{199}{200} \cdot \frac{\Vert q_{+}+q_{-} \Vert_{g}}{2} > \frac{199}{200} \cdot (1-\eta) > \frac{199}{200} \cdot \frac{99}{100}, $$
we obtain
$$ \Vert Ie_{n,m} - Ja \Vert_{Z} < \frac{1}{100} + (h - 1) + \Big( 1 - \frac{199}{200} \cdot \frac{99}{100} \Big) \cdot \frac{200}{199} $$
and
\begin{align*}
\Vert Ie_{n,m} - s \cdot Je_{k,l} \Vert_{Z} & \leq \Vert Ie_{n,m} - Ja \Vert_{Z} + \Vert Ja - s \cdot Je_{k,l} \Vert_{Z} \\
 & < \frac{1}{100} + \frac{1}{199} + \Big( 1 - \frac{199}{200} \cdot \frac{99}{100} \Big) \cdot \frac{200}{199} + \frac{200}{199} \cdot \frac{1}{10} \\
 & < \frac{1}{7}.
\end{align*}

Step 8: Let $ q_{+}+q_{-} $ belong to $ Q_{k,l} $ or $ -Q_{k,l} $. We show that
$$ f(n, m) - g(k, l) < \frac{3}{\delta} \cdot \eta $$
for such $ k $ and $ l $. If we denote $ h' = \alpha + \delta \cdot g(k, l) $, then the elements of $ \pm Q_{k,l} $ fulfill
$$ \Vert w \Vert_{g} = \frac{1}{\sqrt{2}} \cdot h' \cdot |w_{k}+w_{l}| = h' \cdot | \langle w, e_{k,l} \rangle | \leq h' \cdot \Vert w \Vert_{\ell_{2}}, \quad w \in \pm Q_{k,l}. $$
We obtain
$$ 1-\eta < \frac{\Vert q_{+}+q_{-} \Vert_{g}}{2} \leq h' \cdot \frac{\Vert q_{+}+q_{-} \Vert_{\ell_{2}}}{2} < h' \cdot \Big( \frac{1}{h} + \eta \Big), $$
and so
$$ h \cdot (1-\eta) < h' \cdot (1 + h \eta). $$
It follows that
$$ \delta \cdot \big( f(n, m) - g(k, l) \big) = h - h' < h \eta \cdot (1 + h') \leq \frac{200}{199} \cdot \Big( 1 + \frac{200}{199} \Big) \cdot \eta < 3\eta, $$
which provides the desired inequality.
\end{proof}

\begin{proof}[Proof of Lemma~\ref{lemmGHtoK}]
For every $ \{ n, m \} \in [\mathbb{N}]^{2} $, Claim~\ref{claimGHtoK} provides $ \Sigma(n, m) \in [\mathbb{N}]^{2} $ and $ s(n, m) \in \{ -1, 1 \} $ such that
$$ \big\Vert Ie_{n,m} - s(n, m) \cdot Je_{\Sigma(n, m)} \big\Vert_{Z} < \frac{1}{7} $$
and, moreover,
$$ f(n, m) - g(\Sigma(n, m)) < \frac{3}{\delta} \cdot \eta. $$
Let us make a series of observations concerning $ \Sigma(n, m) $ and $ s(n, m) $.

(a) If $ \{ n, m \} $ and $ \{ n', m' \} $ have exactly one common element, then the same holds for $ \Sigma(n, m) $ and $ \Sigma(n', m') $. Indeed, using Lemma~\ref{lemmsep}, we can compute
\begin{align*}
\big\Vert s(n, m & ) \cdot e_{\Sigma(n, m)} - s(n', m') \cdot e_{\Sigma(n', m')} \big\Vert_{\ell_{2}} \\
 & \leq \big\Vert s(n, m) \cdot e_{\Sigma(n, m)} - s(n', m') \cdot e_{\Sigma(n', m')} \big\Vert_{g} \\
 & < \Vert e_{n, m} - e_{n', m'} \Vert_{f} + \frac{1}{7} + \frac{1}{7} \leq \frac{200}{199} \cdot 1 + \frac{1}{7} + \frac{1}{7} < \sqrt{2}, \\
\big\Vert s(n, m & ) \cdot e_{\Sigma(n, m)} - s(n', m') \cdot e_{\Sigma(n', m')} \big\Vert_{\ell_{2}} \\
 & \geq \frac{199}{200} \cdot \big\Vert s(n, m) \cdot e_{\Sigma(n, m)} - s(n', m') \cdot e_{\Sigma(n', m')} \big\Vert_{g} \\
 & > \frac{199}{200} \cdot \Big( \Vert e_{n, m} - e_{n', m'} \Vert_{f} - \frac{1}{7} - \frac{1}{7} \Big) \geq \frac{199}{200} \cdot \Big( 1 - \frac{1}{7} - \frac{1}{7} \Big) > 0,
\end{align*}
and it is sufficient to apply Lemma~\ref{lemmsep} again (in fact, we obtain also $ s(n, m) = s(n', m') $).

(b) If $ \{ n, m \} $ and $ \{ n', m' \} $ are disjoint, then the same holds for $ \Sigma(n, m) $ and $ \Sigma(n', m') $. This can be shown by the same method as above. This time, we have $ \Vert e_{n, m} - e_{n', m'} \Vert_{\ell_{2}} = \sqrt{2} $ and
$$ \big\Vert s(n, m) \cdot e_{\Sigma(n, m)} - s(n', m') \cdot e_{\Sigma(n', m')} \big\Vert_{\ell_{2}} < \frac{200}{199} \cdot \sqrt{2} + \frac{1}{7} + \frac{1}{7} < \sqrt{3}, $$
$$ \big\Vert s(n, m) \cdot e_{\Sigma(n, m)} - s(n', m') \cdot e_{\Sigma(n', m')} \big\Vert_{\ell_{2}} > \frac{199}{200} \cdot \Big( \sqrt{2} - \frac{1}{7} - \frac{1}{7} \Big) > 1. $$

(c) For every $ n \in \mathbb{N} $, there is $ \pi(n) \in \mathbb{N} $ such that $ \pi(n) \in \Sigma(n, m) $ for all $ m \neq n $ (such $\pi(n)$ is unique, but we do not prove it here as it is not needed in what follows). Assume the opposite and pick distinct $ p, q $ different from $ n $. By (a), we can denote the elements of $ \Sigma(n, p) $ and $ \Sigma(n, q) $ by $ a, b, c $ in the way that
$$ \Sigma(n, p) = \{ a, b \}, \quad \Sigma(n, q) = \{ a, c \}. $$
By our assumption, there is $ m \neq n $ such that $ a $ does not belong to $ \Sigma(n, m) $. Then the only possibility for $ \Sigma(n, m) $ allowed by (a) is
$$ \Sigma(n, m) = \{ b, c \}. $$
Pick some $ r $ different from $ n, m, p, q $. Then there is no possibility for $ \Sigma(n, r) $ allowed by (a). Indeed, no set has exactly one common element with all sets $ \{ a, b \}, \{ a, c \} $ and $ \{ b, c \} $.

(d) By (c), there is a function $\pi:\Nat\to\Nat$ with $\pi(n)\in\Sigma(n,m)$ for all $n\neq m$. The function $ \pi $ is injective. Indeed, assume that $ n \neq m $ and pick distinct $ p, q $ different from $ n $ and $ m $. Then $ \pi(n) $ and $ \pi(m) $ belong to the sets $ \Sigma(n, p) $ and $ \Sigma(m, q) $ that are disjoint by (b).

(e) As $ \pi $ is injective, we have $ \Sigma(n, m) = \{ \pi(n), \pi(m) \} $ for all $ \{ n, m \} $, and we can write
$$ \big\Vert Ie_{n,m} - s(n, m) \cdot Je_{\pi(n), \pi(m)} \big\Vert_{Z} < \frac{1}{7} $$
and
$$ f(n, m) - g(\pi(n), \pi(m)) < \frac{3}{\delta} \cdot \eta. $$

(f) Due to the symmetry, there is an injective function $ \xi : \mathbb{N} \to \mathbb{N} $ with the property that
$$ \big\Vert Je_{k, l} - s'(k, l) \cdot Ie_{\xi(k), \xi(l)} \big\Vert_{Z} < \frac{1}{7} $$
and
$$ g(k, l) - f(\xi(k), \xi(l)) < \frac{3}{\delta} \cdot \eta $$
for all $ \{ k, l \} $ and for a suitable $ s'(k, l) \in \{ -1, 1 \} $.

(g) We have $ \pi(\xi(k)) = k $ for every $ k $ and, consequently, $ \pi $ is surjective. Let $ k \in \mathbb{N} $ be given. For every $ l \neq k $, we obtain
\begin{align*}
\big\Vert e_{k, l} - s'( & k, l) s(\xi(k), \xi(l)) \cdot e_{\pi(\xi(k)), \pi(\xi(l))} \big\Vert_{g} \\
 & \leq \big\Vert Je_{k, l} - s'(k, l) \cdot Ie_{\xi(k), \xi(l)} \big\Vert_{Z} \\
 & \quad + | s'(k, l) | \cdot \big\Vert Ie_{\xi(k), \xi(l)} - s(\xi(k), \xi(l)) \cdot Je_{\pi(\xi(k)), \pi(\xi(l))} \big\Vert_{Z} \\
 & < \frac{1}{7} + \frac{1}{7}.
\end{align*}
Due to Lemma~\ref{lemmsep}, this is possible only if $ \{ k, l \} = \{ \pi(\xi(k)), \pi(\xi(l)) \} $ and $ s'(k, l) s(\xi(k), \xi(l)) = 1 $. If we pick distinct $ l_{1} $ and $ l_{2} $ different from $ k $, then $ k \in \{ \pi(\xi(k)), \pi(\xi(l_{1})) \} \cap \{ \pi(\xi(k)), \pi(\xi(l_{2})) \} = \{ \pi(\xi(k)) \} $.

(h) We check that $ \pi $ works. We already know that $ \pi \in S_{\infty} $ and $ \pi^{-1} = \xi $. Thus, we obtain
$$ g(\pi(n), \pi(m)) - f(n, m) = g(\pi(n), \pi(m)) - f(\xi(\pi(n)), \xi(\pi(m))) < \frac{3}{\delta} \cdot \eta $$
for every $ \{ n, m \} \in [\mathbb{N}]^{2} $. Finally, combining this with an above inequality,
$$ \big| g(\pi(n), \pi(m)) - f(n, m) \big| < \frac{3}{\delta} \cdot \eta, $$
which completes the proof of the lemma.
\end{proof}

\begin{proof}[Proof of Theorem \ref{thmGHtoK}]
During the proof, we make no difference between a metric $ f \in \Met_{1/2}^{1} $ and the corresponding function $ f : [\mathbb{N}]^{2} \to [1/2, 1] $. For $ f \in \Met_{1/2}^{1} $, we can thus consider the norm $ \Vert \cdot \Vert_{f} $ defined above. It is clear that there is an injective Borel mapping from $ \Met_{1/2}^{1} $ into $ \Banach $ such that the image of $ f $ is isometric to $ (\ell_{2}, \Vert \cdot \Vert_{f}) $ (it is sufficient to restrict the norm $ \Vert \cdot \Vert_{f} $ to $ V $).

To prove the first part of the theorem, we show a series of inequalities that illustrates that the Gromov-Hausdorff distance of $ M_{f} $ and $ M_{g} $ and all the involved distances between $ (\ell_{2}, \Vert \cdot \Vert_{f}) $ and $ (\ell_{2}, \Vert \cdot \Vert_{g}) $ are uniformly equivalent.

(1) We show that
$$ \rho_{BM} \big( (\ell_{2}, \Vert \cdot \Vert_{f}), (\ell_{2}, \Vert \cdot \Vert_{g}) \big) \leq C \rho_{GH}(M_{f}, M_{g}) $$
for every $ f, g \in \Met_{1/2}^{1} $. If $ \rho_{GH}(M_{f}, M_{g}) \geq \frac{1}{4} $, then
$$ \rho_{BM} \big( (\ell_{2}, \Vert \cdot \Vert_{f}), (\ell_{2}, \Vert \cdot \Vert_{g}) \big) \leq 2 \log \Big( \frac{200}{199} \Big) \leq 2 \log \Big( \frac{200}{199} \Big) \cdot 4 \rho_{GH}(M_{f}, M_{g}). $$
Assuming $ \rho_{GH}(M_{f}, M_{g}) < \frac{1}{4} $, we pick $ r $ with $ \rho_{GH}(M_{f}, M_{g}) < r < \frac{1}{4} $. Since $ f, g \in \Met_{p} $ and $ \rho_{GH}(M_{f}, M_{g}) < p/2 $ for $ p = 1/2 $, Lemma~\ref{lem:GHequivalenceUniformlyDiscrete} provides $ \pi \in S_{\infty} $ such that
$$ \big| g(\pi(n), \pi(m)) - f(n, m) \big| \leq 2r, \quad \{ n, m \} \in [\mathbb{N}]^{2}. $$
Let us consider the linear isometry $ T : \ell_{2} \to \ell_{2} $ which maps $ e_{n} $ to $ e_{\pi(n)} $. For $ x \in \ell_{2} $, we have
\begin{align*}
\Vert Tx \Vert_{g} & = \sup \Big( \{ \Vert Tx \Vert_{\ell_{2}} \} \cup \Big\{ \frac{1}{\sqrt{2}} \cdot \big( \alpha + \delta \cdot g(k, l) \big) \cdot |(Tx)_{k}+(Tx)_{l}| : k \neq l \Big\} \Big) \\
 & = \sup \Big( \{ \Vert x \Vert_{\ell_{2}} \} \cup \Big\{ \frac{1}{\sqrt{2}} \cdot \big( \alpha + \delta \cdot g(\pi(n), \pi(m)) \big) \cdot |x_{n}+x_{m}| : n \neq m \Big\} \Big),
\end{align*}
and so
\begin{align*}
| \Vert Tx \Vert_{g} - \Vert x \Vert_{f} | & \leq \sup \Big\{ \frac{1}{\sqrt{2}} \cdot \delta \cdot \big| g(\pi(n), \pi(m)) - f(n, m) \big| \cdot |x_{n}+x_{m}| : n \neq m \Big\} \\
 & \leq \delta \cdot 2r \cdot \Vert x \Vert_{\ell_{2}}.
\end{align*}
It follows that $ \Vert Tx \Vert_{g} \leq (1+2\delta r)\Vert x \Vert_{f} $ and $ \Vert x \Vert_{f} \leq (1+2\delta r)\Vert Tx \Vert_{g} $. We obtain $ \rho_{BM}((\ell_{2}, \Vert \cdot \Vert_{f}), (\ell_{2}, \Vert \cdot \Vert_{g})) \leq 2 \log (1+2\delta r) \leq 2 \cdot 2\delta r $. As $ r $ could be chosen arbitrarily close to $ \rho_{GH}(M_{f}, M_{g}) $, we arrive at
$$ \rho_{BM} \big( (\ell_{2}, \Vert \cdot \Vert_{f}), (\ell_{2}, \Vert \cdot \Vert_{g}) \big) \leq 4\delta \rho_{GH}(M_{f}, M_{g}). $$
Therefore, the choice $ C = \max \{ 8 \log (\frac{200}{199}), 4\delta \} $ works.

(2) It is easy to check that
\begin{align*}
2\rho_{N} \big( (\ell_{2}, \Vert \cdot \Vert_{f}) & , (\ell_{2}, \Vert \cdot \Vert_{g}) \big) \leq \rho_{U} \big( (\ell_{2}, \Vert \cdot \Vert_{f}), (\ell_{2}, \Vert \cdot \Vert_{g}) \big) \\
 & \leq 2\rho_{L} \big( (\ell_{2}, \Vert \cdot \Vert_{f}), (\ell_{2}, \Vert \cdot \Vert_{g}) \big) \leq \rho_{BM} \big( (\ell_{2}, \Vert \cdot \Vert_{f}), (\ell_{2}, \Vert \cdot \Vert_{g}) \big)
\end{align*}
for every $ f, g \in \Met_{1/2}^{1} $.

(3) By \cite[Proposition~2.1]{DuKa}, we have
$$ \rho_{GH}^{\Banach} \big( (\ell_{2}, \Vert \cdot \Vert_{f}), (\ell_{2}, \Vert \cdot \Vert_{g}) \big) \leq e^{2\rho_{N} ((\ell_{2}, \Vert \cdot \Vert_{f}), (\ell_{2}, \Vert \cdot \Vert_{g}))} - 1 $$
for every $ f, g \in \Met_{1/2}^{1} $.

(4) There is a function $ \varphi : (0, 1] \to (0, 1] $ with $ \lim_{\varepsilon \to 0} \varphi(\varepsilon) = 0 $ such that
$$ \rho_{K} \big( (\ell_{2}, \Vert \cdot \Vert_{f}), (\ell_{2}, \Vert \cdot \Vert_{g}) \big) \leq C \varphi \Big( \rho_{GH}^{\Banach} \big( (\ell_{2}, \Vert \cdot \Vert_{f}), (\ell_{2}, \Vert \cdot \Vert_{g}) \big) \Big) $$
for every $ f, g \in \Met_{1/2}^{1} $. Considering any $ 0 < r < 1 $, a function provided by \cite[Theorem~3.6]{KalOst} (denoted $ f $ there) works. Indeed, if we adopt some notation from \cite{KalOst}, then \cite[Theorem~3.7]{KalOst} provides
$$ \rho_{K} \big( (\ell_{2}, \Vert \cdot \Vert_{f}), (\ell_{2}, \Vert \cdot \Vert_{g}) \big) \leq C(r) \cdot \kappa_{0} \big( (\ell_{2}, \Vert \cdot \Vert_{f}) \big) \cdot d_{r} \big( (\ell_{2}, \Vert \cdot \Vert_{f}), (\ell_{2}, \Vert \cdot \Vert_{g}) \big) $$
$$ \leq C(r) \cdot \frac{200}{199} \cdot \kappa_{0}(\ell_{2}) \cdot \varphi \Big( \rho_{GH}^{\Banach} \big( (\ell_{2}, \Vert \cdot \Vert_{f}), (\ell_{2}, \Vert \cdot \Vert_{g}) \big) \Big). $$

(5) We show that $$ \rho_{GH}(M_{f}, M_{g}) \leq C \rho_{K} \big( (\ell_{2}, \Vert \cdot \Vert_{f}), (\ell_{2}, \Vert \cdot \Vert_{g}) \big) $$
for every $ f, g \in \Met_{1/2}^{1} $. Let us denote
$$ d = \rho_{K} \big( (\ell_{2}, \Vert \cdot \Vert_{f}), (\ell_{2}, \Vert \cdot \Vert_{g}) \big), \quad \eta_{max} = \min \Big\{ \frac{1}{100}, \frac{1}{10} \cdot \frac{\sqrt{\alpha^{2}-1}}{\alpha}, \frac{1}{2} \Big( 1 - \frac{1}{\alpha} \Big) \Big\}. $$
If $ d \geq \eta_{max} $, then
$$ \rho_{GH}(M_{f}, M_{g}) \leq 1 \leq \frac{1}{\eta_{max}} \cdot d. $$
Assuming $ d < \eta_{max} $, we pick $ d < \eta < \eta_{max} $. Then Lemma~\ref{lemmGHtoK} can be applied, and we obtain $ f \simeq_{2\varepsilon} g $ for $ \varepsilon = \frac{1}{2} \cdot \frac{3}{\delta} \cdot \eta $. By Lemma~\ref{lem:lehciImplikace}, we get $ \rho_{GH}(M_{f}, M_{g}) \leq \frac{1}{2} \cdot \frac{3}{\delta} \cdot \eta $. As $ \eta $ can be chosen arbitrarily close to $ d $, we arrive at
$$ \rho_{GH}(M_{f}, M_{g}) \leq \frac{1}{2} \cdot \frac{3}{\delta} \cdot d. $$
It follows that the choice $ C = \max \{ \frac{1}{\eta_{max}}, \frac{3}{2\delta} \} $ works.

Finally, concerning the moreover part of the theorem, it remains to notice that
$$ \rho_{K} \big( (\ell_{2}, \Vert \cdot \Vert_{f}), (\ell_{2}, \Vert \cdot \Vert_{g}) \big) \leq \rho_{BM} \big( (\ell_{2}, \Vert \cdot \Vert_{f}), (\ell_{2}, \Vert \cdot \Vert_{g}) \big) $$
by \cite[Proposition~6.2]{o94} (or \cite[Proposition~2.1]{DuKa}), and it is sufficient to use the inequalities proven in (1) and (5).
\end{proof}

\section{Concluding remarks}\label{section:problems}
The following diagram summarizes the reducibility results we have proved in this paper and includes also the reducibility results proved in \cite{CDKpart1}. By $\rho_{E_G}$ we denote the pseudometric induced by the universal orbit equivalence relation $E_G$ discussed in \cite[Section 5]{CDKpart1}; by $\rho_{S_\infty, d}$ we denote the CTR orbit pseudometric given in \cite[Section 3, Example 2]{CDKpart1},  by $\rho_{univ}$ we mean the universal analytic pseudometric which exists by \cite[Theorem 6]{CDKpart1}; all the remaining pseudometrics are explained in this paper. The reducibilities which are not explicitly mentioned in the diagram are not known to us. 

\bigskip

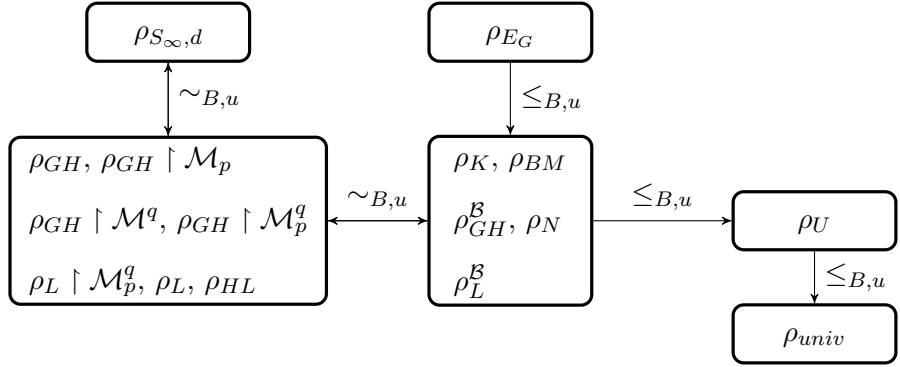
\begin{tikzpicture}[->,>=stealth']

\tikzset{
    state/.style={
           rectangle,
           rounded corners,
           draw=black, very thick,
           minimum height=2em,
           inner sep=2pt,
           text centered,
           },
}
 
 \node[state] (prvniBox) 
 {\begin{tabular}{l}
 	$\rho_{GH}$, $\rho_{GH}\upharpoonright \Met_p$\\[10pt]
    $\rho_{GH}\upharpoonright \Met^q$, $\rho_{GH}\upharpoonright \Met_p^q$\\[10pt]
    $\rho_L\upharpoonright \Met_p^q$, $\rho_L$, $\rho_{HL}$
 \end{tabular}};
 
   \node[state,    	
  text width=2cm, 	
  left of=prvniBox, 
  yshift=2.5cm,
  node distance=0cm,
  anchor=center] (nultyBox) 
 {\begin{tabular}{l}
 	$\rho_{S_\infty, d}$    
 \end{tabular}};
 
 \node[state,    	
  text width=2cm, 	
  right of=prvniBox, 	
  node distance=4.5cm, 	
  anchor=center] (druhyBox) 	
 {
 \begin{tabular}{l} 
  $\rho_K$, $\rho_{BM}$\\[10pt]
  $\rho_{GH}^\Banach$, $\rho_N$\\[10pt]
  $\rho_L^\Banach$
 \end{tabular}
 };
 
 \node[state,    	
  text width=2cm, 	
  left of=druhyBox, 
  yshift=2.5cm,
  node distance=0cm,
  anchor=center] (zapornyBox) 
 {\begin{tabular}{l}
 	$\rho_{E_G}$    
 \end{tabular}};
 
  \node[state,
  text width=2cm,
  right of=druhyBox,
  node distance=4cm,
  anchor=center] (tretiBox)
 {
 \begin{tabular}{l} 	
  $\rho_{U}$
 \end{tabular}
 };
 
  \node[state,    	
  text width=2cm, 	
  right of=tretiBox,
  yshift=-1.5cm,
  node distance=0cm,
  anchor=center] (ctvrtyBox) 
 {\begin{tabular}{l}
 	$\rho_{univ}$    
 \end{tabular}};
 
 \path (druhyBox) edge  node[anchor=north,above]{$\leq_{B,u}$} (tretiBox)
    (nultyBox) edge  node[anchor=north,right]{$\sim_{B,u}$} (prvniBox)
    (tretiBox) edge  node[anchor=north,right]{$\leq_{B,u}$} (ctvrtyBox)
    (druhyBox) edge node[anchor=north,above]{$\sim_{B,u}$} (prvniBox)
    (zapornyBox) edge node[anchor=north,right]{$\leq_{B,u}$} (druhyBox)
    (prvniBox) edge (druhyBox)
    (prvniBox) edge (nultyBox);
     
 \end{tikzpicture}

\bigskip

We believe there is enough space for investigating other reductions. The interested reader can find many more distances for which their exact place in the reducibility diagram is not known. This includes the uniform distance $\rho_U$, or distances that we mentioned in \cite[Section 3]{CDKpart1} but left untouched, such as the completely bounded Banach-Mazur distance or e.g. the orbit version of the Kadison-Kastler distance.\\

\noindent{\bf Acknowledgements}

\medskip
M. C\' uth was supported by Charles University Research program No. UNCE/SCI/023 and by the Research grant GA\v CR 17-04197Y. M. Doucha was supported by the GA\v CR project EXPRO 20-31529X, and RVO: 67985840. O. Kurka was supported by the Research grant GA\v CR 17-04197Y and by RVO: 67985840.

\end{document}